\documentclass{amsart}
\usepackage{geometry}                
\usepackage{graphicx}
\usepackage{amssymb}
\usepackage{epstopdf,amsmath}
\usepackage[shortlabels]{enumitem}

\usepackage{xcolor}

\newtheorem{prop}{Proposition}[section]
\newtheorem{lemma}[prop]{Lemma}

\newtheorem{remark}[prop]{Remark}
\newtheorem{cor}[prop]{Corollary}
\newtheorem{theorem}[prop]{Theorem}

\newtheorem{definition}[prop]{Definition}

\newcommand {\ignore} [1] {}

\newcommand{\N}{\mathbb{N}}
\newcommand{\Z}{\mathbb{Z}}

\newcommand{\R}{\mathbb{R}}

\newcommand{\F}{\mathbb{F}}
\newcommand{\T}{\mathcal{T}}
\newcommand{\U}{\mathbb{U}}

\newcommand{\eps}{\varepsilon}
\newcommand{\leapp}{\lessapprox}
\newcommand{\geapp}{\gtrapprox}
\newcommand{\app}{\approx}
\newcommand{\lesim}{\lesssim}
\newcommand{\gesim}{\gtrsim}

\newcommand{\bp}{\mathbf{p}}
\newcommand{\bT}{\mathbf{T}}

\newcommand{\cD}{\mathcal{D}}

\newcommand{\cB}{\mathcal{B}}
\newcommand{\cS}{\mathcal{S}}
\newcommand{\cT}{\mathcal{T}}

\newcommand{\cQ}{\mathcal{Q}}

\newcommand{\cP}{\mathcal{P}}

\DeclareMathOperator{\dir}{dir}
\newcommand{\Tbad}{\T_{\mathrm{bad}}}

\newcommand{\bD}{\mathbf{D}}

\title{Furstenberg sets estimate in the plane}

\author{Kevin Ren}

\address{Department of Mathematics, Princeton University}

\email{kevinren@princeton.edu}

\author{Hong Wang}

\address{Courant institute of mathematical sciences, New York University}

\email{hw3639@nyu.edu}

\begin{document}

\maketitle

\begin{abstract}
    We fully resolve the Furstenberg set conjecture in $\mathbb{R}^2$, that a $(s, t)$-Furstenberg set has Hausdorff dimension $\ge \min(s+t, \frac{3s+t}{2}, s+1)$. As a result, we obtain an analogue of Elekes' bound for the discretized sum-product problem and resolve an orthogonal projection question of Oberlin.
\end{abstract}











\section{Introduction}
For $s \in (0, 1]$ and $t \in (0, 2]$, a $(s,t)$-Furstenberg set is a set $E\subset\mathbb{R}^2$ with the following property: there exists a family $\mathcal{L}$ of lines with $\dim_H \mathcal{L}\geq t$ such that $\dim_H(E\cap \ell)\geq s$ for all $\ell \in \mathcal{L}$. The Hausdorff dimension $\dim_H\mathcal{L}$ is defined as the Hausdorff dimension of $\mathcal{L}$ in the space $\mathcal{A}(2,1)$ of all affine lines in $\mathbb{R}^2$ (see \cite[Section 3.16]{mattila1999geometry}). The Furstenberg set question asks: what is the minimum possible dimension of $\dim_H (E)$?

Motivated by his work on $\times 2, \times 3$-invariant sets \cite{furstenberg1970intersections}, Furstenberg considered a set $E$ with the following property: there exists a family of lines $\mathcal{L}$, one in every direction, such that $\dim_H(E\cap \ell)\geq s$ for all $\ell \in \mathcal{L}$. He conjectured (in unpublished work, but recalled by Wolff in \cite[Remark 1.5]{wolff1999recent}) that such a set $E$ (which is a special $(s, 1)$-Furstenberg set) has Hausdorff dimension $\ge \frac{3s+1}{2}$. Furstenberg's conjecture was later generalized to $(s, t)$-Furstenberg sets: any $(s,t)$-Furstenberg set $E\subset \mathbb{R}^2$ has Hausdorff dimension
\begin{equation}
    \dim_H E \geq \min \{ s+t, \frac{3s+t}{2}, s+1\}.
\end{equation}
As observed by Wolff \cite{wolff1999recent}, the conjectured bound $\frac{3s+t}{2}$ is optimal  for the  continuum version of a sharp example for Szemer\'edi-Trotter theorem \cite{szemeredi1983extremal}.
The bound $s+t$ for $0 \leq t\leq s \leq 1$ was proved by Lutz-Stull \cite{lutz2020bounding} and H\'era-Shmerkin-Yavicoli \cite{hera2021improved},  and the bound $s+1$ for $s+t\geq 2$ was proved recently by Fu and the first author \cite{fu2021incidence}. Wolff \cite{wolff1999recent} gave a sharp construction for $\frac{3s+t}{2}$ in the case $t = 1$ based on a grid example (which easily generalizes to other values of $t$), the sharp construction for $s+t$ can be taken as the product of a $t$-dim Cantor set and an $s$-dim Cantor set, and Fu and the first author \cite{fu2021incidence} gave a sharp construction for $s+1$ (the product of an interval and a $s$-dim Cantor set).

The main  case $s<t<2-s$ remains open. Elementary bounds due to Wolff \cite{wolff1999recent}  in the case $t = 1$,  Molter-Rela \cite{molter2012furstenberg}, H\'{e}ra \cite{hera2019hausdorff}, and Lutz-Stull \cite{lutz2020bounding} give 
\[
\dim_H E \geq \max(t/2+s, s + \min(s, t)). 
\]
Meanwhile, an $\eps$-improvement $\dim_H (E) \ge 2s+\eps(s, t)$ was proven in the special case $(s, t) = (1/2, 1)$ by Katz-Tao \cite{katz2001some} and Bourgain \cite{bourgain2003erdHos}, then for general $(s, 2s)$ by H\'{e}ra-Shmerkin-Yavicoli \cite{hera2021improved}.  Explicit improvement  for $(s, 2s)$-Furstenberg sets  was obtained by Benedetto-Zahl \cite{di2021new}. 

A breakthrough by Orponen-Shmerkin \cite{orponen2021hausdorff} proves an $\eps$-improvement over $2s$ for general $(s, t)$ with $s < t$ (see Section~\ref{section: roadmap} for how this result leads to recent progress on several problems). Later, work of Shmerkin and the second author \cite{shmerkin2022dimensions} provides an $\eps$-improvement over the bound $s + \frac{t}{2}$ using ideas in \cite{orponen2021hausdorff}.  Very recently, Orponen and Shmerkin \cite{orponen2023projections} made another breakthrough and proved $\dim_H E \geq \max(t/2 +s, s + \min(s, t)) + C(t,s)$ for some good explicit constant $C(t,s)$ depending only on $s$ and $t$ which is not far from the conjecture. In this paper, we prove the Furstenberg conjecture on the Hausdorff dimension of $(s,t)$-Furstenberg sets.

\begin{theorem}\label{thm: Furstenberg}
    Fix $s \in (0, 1]$ and $t \in (0, 2]$. Any $(s, t)$-Furstenberg set $F \subset \R^2$ has Hausdorff dimension
    \begin{equation*}
        \dim_H F \ge \min\{s+t, \frac{3s+t}{2}, s+1\}.
    \end{equation*}
\end{theorem}
The $(s,t)$-Furstenberg problem is known to be connected to the discretized sum-product problem and projection problem, as shown in Section~\ref{section: orthogonalproj}, Section~\ref{section: sumproduct},  and \cite[Section 6]{orponen2023projections}. 

Recently, Wang and Wu \cite{wang2024restriction} proved a ``two-ends'' generalization of our Theorem \ref{thm: Furstenberg} and used it with Bourgain-Demeter decoupling theory to make progress on Stein's Fourier restriction conjecture. Their proof uses Theorem~\ref{thm: Furstenberg} as a black box and combines it with multi-scale analysis and induction on scales; we refer the reader to that paper and \cite{demeter2024szemer} for more details.


\subsection{Exceptional set estimate for orthogonal projections}\label{section: orthogonalproj}
 Let $K \subset \R^2$ be an arbitrary Borel set and $0 \le u \le \min \{ \dim_H K, 1\}$. The exceptional set problem of orthogonal projections asks for the optimal upper bound on 
\[
 \dim_H \{ \theta \in S^1 : \dim_H \pi_\theta (K) < u \}.
\]
In 1968, 
Kaufman \cite{kaufman1968hausdorff} proved that 
\[
  \dim_H \{ \theta \in S^1 : \dim_H \pi_\theta (K) < u \} \le u, 
\]
sharpening a result of Marstrand \cite{marstrand1954some}. 
In 2000,  Peres-Schlag \cite{peres2000smoothness} showed  that when $\dim_H K >1 $, one can improve the estimate to 
\[\dim_H \{ \theta \in S^1 : \dim_H \pi_\theta (K) < u \} \le \max\{ u +1-\dim_H K , 0\}.\]
Kaufman's bound is sharp when $u=\dim_H K$ and the Peres-Schlag bound is sharp when $u=1$. In fact, the case $u=1$ was obtained previously by Falconer \cite{falconer1982hausdorff}.   For general cases, Oberlin \cite{oberlin2012restricted} conjectured that the right-hand side should be $ \max\{ 2u - \dim_H K, 0 \}.$

Further improvements on the exceptional set estimate originate from Bourgain's projection theorem \cite{bourgain2010discretized},  which says
\begin{equation}\label{eq: Bourgainproj}
  \dim_H \{ \theta \in S^1 : \dim_H \pi_\theta (K) \leq \dim_H K/2 \}=0. 
\end{equation}

Bourgain's projection theorem concerns the deep nature of real numbers and has been used in Bourgain-Furman-Lindenstrauss-Mozes \cite{bourgain2011stationary} on random walk on torus  and in Katz-Zahl \cite{katz2019improved} on the Hausdorff dimension of Kakeya sets in $\mathbb{R}^3$. Recent improvements  \cite{orponen2021hausdorff, orponen2023projections} on the exceptional set estimate are associated to the improvements on the Furstenberg sets problem.

As a corollary of Theorem~\ref{thm: Furstenberg} (in fact, its discretized version Theorem~\ref{Thm: main}), we prove Oberlin's conjecture \cite[(1.8)]{oberlin2012restricted} on the exceptional sets of the orthogonal projections, whose AD-regular version was established in \cite[Theorem 1.13]{orponen2023projections}. This deduction is analogous to the procedure used to establish \cite[Theorem 1.13]{orponen2023projections} (see \cite[Section 6.1]{orponen2023projections}), so we omit the details.


\begin{theorem}\label{thm:projections}
    Let $K \subset \R^2$ be an arbitrary Borel set. Then for all $0 \le u \le \min \{\dim_H (K), 1\}$, we have
    \begin{equation*}
        \dim_H \{ \theta \in S^1 : \dim_H \pi_\theta (K) < u \} \le \max\{ 2u - \dim_H K, 0 \}.
    \end{equation*}
\end{theorem}

\subsection{Discretized sum-product}\label{section: sumproduct}

Let  $\mathbb{F}$ be a field and $A, B\subset \mathbb{F}$, the sum set $A+B$ is defined as 
\[
A+B:=\{ a+b: a\in A, b\in B\}
\]
and the  product set is defined as 
\[
A\cdot B:=\{ a\cdot b: a\in A, b\in B\}.
\]
The sum-product problem asks whether the size of  $A+A$ or $A\cdot A$  can be substantially larger than the size of $A$.  When $A$ is a finite set, we measure the size of $A$ by its cardinality and it was conjectured by Erd\"{o}s and Szemer\'{e}di that when $\mathbb{F}=\mathbb{R}$, for any $\epsilon>0$, there exists $C_{\epsilon}>0$ such that 
\[
\max\{ |A+A|, |A\cdot A|\}\geq C_{\epsilon} |A|^{2-\epsilon}. 
\]
 See \cite{solymosi2009bounding, mohammadi2023attaining,roche2023better,roche2016new,rudnev2020stronger,rudnev2022update} for some recent results on the Erd\"{o}s-Szemer\'{e}di sum-product conjecture. 

 When $\mathbb{F}=\mathbb{R}$, a ``fractal'' analogue of the discrete sum-product problem concerns the $(\delta, s, C)$-sets, which are discretizations of  sets of Hausdorff dimension $\geq s$ in $\mathbb{R}$. For a bounded set $A \subset \R^d$, let $|A|_\delta$ denote the dyadic $\delta$-covering number of $A$. This concept will be precisely defined in Definition \ref{defn:covering number}, but for now one can regard $|A|_\delta$ as being the same (up to a constant factor) as the least number of $\delta$-balls needed to cover $A$.

 \begin{definition}[$(\delta, s, C)$-set]\label{defn:delta, s-set}
     For $\delta\in 2^{-\mathbb{N}}$, $s\in [0,d]$ and $C>0$, a non-empty bounded set $A\subset \mathbb{R}^d$ is called a $(\delta, s, C)$-set  if 
     \[
     |A\cap B(x,r)|_{\delta} \leq C r^{s} |A|_{\delta}, \quad \forall x\in \mathbb{R}^d, r\in [\delta, 1].
     \]
     \end{definition}
If $A$ is a $(\delta, s, C)$-set, then $|A|_{\delta}\geq C^{-1} \delta^{-s}$ and if $A'\subset A$ satisfies $|A'|_{\delta} \geq C'^{-1} |A|_{\delta}$, then $A'$ is a $(\delta, s, CC')$-set. 

Definition~\ref{defn:delta, s-set} is to be contrasted with the notion of a $(\delta, s, C)$-Katz-Tao set, which imposes the bound
\[
     |A\cap B(x,r)|_{\delta} \leq C \left( \frac{r}{\delta} \right)^{s}, \quad \forall x\in \mathbb{R}^d, r\in [\delta, 1].
     \]

     \begin{definition}[AD-regular set]\label{def: AD}
      A $(\delta,s, C)$-set  $A$ is called  Alfors-David regular  at scale $\delta$  if  $C>0$ is a constant independent of $\delta$ and 
      \[
     C^{-1} r^{s} |A|_{\delta}\leq |A\cap B(x,r)|_{\delta} \leq C r^{s} |A|_{\delta}, \quad \forall x\in A, r\in [\delta, 1].
     \]
  We say that $A$ is $(s, \epsilon)$-almost AD-regular if $C=\delta^{-\epsilon}$. Note that in the above condition, we have $x \in A$, not $x \in \R^d$!
 \end{definition}

 
Katz and Tao \cite{katz2001some} made the following discretized ring conjecture and showed that it is equivalent to an $\epsilon$-improvement of the $(1/2, 1)$-Furstenberg set problem and Falconer distance set problem for sets of Hausdorff dimension $1$ on the plane.    When $\eps>0$ is sufficiently small, if $A\subset [1,2]$ is a $(\delta, s, \delta^{-\eps})$-set with $s=1/2$, then 
\[
\max\{|A+A|_{\delta}, |A\cdot A|_{\delta}\} \geq C_{s, \eps} \delta^{-s-\eps}. 
\]
This conjecture was proved in 2003 by Bourgain  \cite{bourgain2003erdHos}. (In related work, Edgar and Miller \cite{edgar2003borel} proved the qualitative statement that there do not exist $s$-dimensional Borel subrings of $\R$ for any $s\in (0,1)$).  Later Bourgain generalized the conjecture to the aforementioned projection estimate \eqref{eq: Bourgainproj}. More recently, Guth-Katz-Zahl \cite{guth2021discretized} used an elementary method inspired by Gaarev's approach to the sum-product problem in finite fields to prove a quantitative result, replacing $\epsilon$ by an explicit constant depending on $s$.

As a corollary of Theorem~\ref{thm:projections}, we prove  a new estimate for the discretized sum-product problem, which complements the analogous results of Elekes \cite{elekes1997number} for $\R$ and Mohammadi-Stevens \cite{mohammadi2023attaining} for finite fields $\F_p$. This deduction is analogous to the procedure used to establish \cite[Theorem 1.26]{orponen2023projections} (see \cite[Section 5.2]{orponen2023projections}), so we omit the details.

\begin{theorem}\label{thm:sum-product}
    Given $s \in (0, \frac{2}{3})$ and $\eta > 0$, there is $\eps = \eps(s, \eta) > 0$ such that the following holds for small enough $\delta > 0$. Let $A \subset [1, 2]$ be a $(\delta, s, \delta^{-\eps})$-set. Then
    \begin{equation*}
        \max\{ |A + A|_\delta, |A \cdot A|_\delta \} \ge \delta^{-(5/4)s + \eta}.
    \end{equation*}
\end{theorem}
Theorem~\ref{thm:sum-product} improves the very recent result of M\'{a}th\'{e}-O'Regan \cite[Theorem B]{mathe2023discretised} and Orponen-Shmerkin \cite[Theorem 1.26]{orponen2023projections} where $5/4$ was replaced by $7/6$ and extends \cite[Theorem 1.26]{orponen2023projections} for almost AD-regular sets to general $(\delta, s, C)$-sets. 

\begin{remark}
    If $s \ge \frac{2}{3}$, then \cite[Corollary 1.7]{fu2021incidence} gives the lower bound $\delta^{-\frac{s+1}{2} + \eta}$. This bound is sharp, see \cite[Remark 1.13]{mathe2023discretised}.
\end{remark}

While the sum-product theorem can be viewed as a special case of the Furstenberg sets problem, 
the asymmetric sum product, i.e. $A+B\cdot C$ for different sets $A, B, C\subset [0,1]$, plays an important role in Orponen and Shmerkin's recent work \cite{orponen2023projections} on the Furstenberg set conjecture.  


\subsection{Roadmap}\label{section: roadmap}
Our proof builds on previous works on projection problems and incidences for tubes. We illustrate the relation between different results below. 
\begin{itemize}[,]
    \item Bourgain's discretized sum-product theorem \cite{bourgain2003erdHos}/Bourgain's projection theorem \cite{bourgain2010discretized}
    \item $\Longrightarrow$ $\epsilon$-improvement of Furstenberg sets estimate \cite{orponen2021hausdorff}
    \item $\Longrightarrow$ Sharp radial projection theorem in the plane \cite{orponen2022kaufman} 
    \item $\Longrightarrow$ $ABC$-sum product theorem \cite{orponen2023projections}, \cite{eberhard2023ABC}
    \item $\Longrightarrow$ Sharp exceptional sets estimate for almost AD-regular sets \cite{orponen2023projections}
    \item $\Longrightarrow$ Sharp Furstenberg sets estimate for almost AD-regular sets \cite{orponen2023projections}
    \item + Sharp Furstenberg sets estimate for well-spaced sets \cite{guth2019incidence}
    \item $\Longrightarrow$ Sharp Furstenberg sets estimate.
\end{itemize}








In \cite{orponen2023projections}, the Furstenberg conjecture is solved for almost AD-regular sets. On the other hand, in \cite{guth2019incidence}, the Furstenberg conjecture is solved for well-spaced sets (see also \cite{fu2022incidence} for a related result). In this paper, we will generalize the main result in \cite{guth2019incidence} and solve the Furstenberg conjecture for a class of semi-well spaced sets. Then we combine this result with the almost AD-regular result in an induction-on-scales scheme due to \cite{orponen2021hausdorff} to finish the proof of  Furstenberg sets conjecture.

It is interesting to note that the proof of  Furstenberg sets conjecture for almost AD-regular sets completely relies on tools from geometric measure theory and additive combinatorics, while our result for semi-well spaced sets completely relies on Fourier analysis, specifically the high-low method pioneered by
Vinh \cite{vinh2011szemeredi} and Guth, Solomon, and the second author \cite{guth2019incidence}.

\subsection{Structure of the paper}
In Section 2, we introduce preliminaries and notations about dyadic tubes, dimensions of sets, and the high-low method. In Section 3, we state and prove some pigeonholing and multiscale lemmas that will be crucial later in the proof. In Section 4, we re-prove the special case $s + t = 2$ of the Furstenberg set problem, as well as introduce the key Lemma \ref{lem:r-rich} that will allow us to solve the Furstenberg set problem for a class of semi-well spaced sets. In Section 5, we prove Lemma \ref{lem:r-rich} using a high-low method. Finally, in Section 6, we prove a decomposition lemma for Lipschitz functions that reduces the full Furstenberg set conjecture to the AD-regular and semi-well spaced cases.

\subsection{Acknowledgments} The first author is supported by an NSF GRFP fellowship. The second author is supported by NSF CAREER DMS-2238818 and NSF DMS-2055544. We would like to thank Larry Guth, Pablo Shmerkin, and Tuomas Orponen for helpful discussions and  for careful reading of an earlier draft.  Finally, the first author would like to thank Yuqiu Fu for introducing him to the Furstenberg set problem.

\section{preliminaries}

In this section, we recall some definitions and lemmas from \cite{orponen2023projections} and the high-low method.

The letter $C$ will usually denote a constant, while the letter $K$ will usually denote a parameter depending on $\delta$ (e.g. $K = C \delta^{-\eps}$).

We say $A \lesim B$ if $A \le CB$ for some constant $C$, and define $A \gesim B$, $A \sim B$ similarly. We write $A \lesim_n B$ to emphasize the implicit constant may depend on $n$.

For a finite set $A$, $|A|$ denotes cardinality of $A$. For a measurable and infinite set $A$ (usually a ball or a rectangle), $|A|$ denotes the Lebesgue measure of $A$.
 
For any compact set $A\in \mathbb{R}^d$ and $r > 0$, let $A^{(r)}:= \{ x \in \R^2 : d(x, A) \le r \}$ denote the $r$-neighborhood of $A$. 

\begin{definition}[Dyadic cubes]
    If $n\in \mathbb{Z}$ and $\delta=2^{-n}$, denote by $\cD_{\delta}(\mathbb{R}^d) := \{ [a\delta, (a+1)\delta) \times [b\delta, (b+1)\delta) : a, b \in \Z \}$ the family of dyadic cubes  in $\mathbb{R}^d$ of side-length $\delta$. If $\cP\subset \mathbb{R}^d$, denote 
    \[
    \cD_{\delta}(\cP):=\{ Q\in \cD_{\delta} (\R^d): Q\cap \cP\neq \emptyset\}.
    \]
    We abbreviate $\cD_{\delta}:=\cD_{\delta}([0,1)^d)$ for $n\geq 0$.  If $\delta <\Delta\in 2^{-\mathbb{N}}$ and $p\in \cD_{\delta}$, let $p^{\Delta}$ denote the dyadic cube $\bp \in \cD_{\Delta}$ containing $p$. 
\end{definition}

\begin{definition}\label{defn:covering number}
    For a bounded set $P \subset \R^d$ and $\delta = 2^{-n}$, $n \ge 0$, define the dyadic $\delta$-covering number
    \begin{equation*}
        |P|_\delta := |\cD_\delta (P)|.
    \end{equation*}
\end{definition}
As remarked before, $|P|_{\delta}$ is comparable (up to a universal constant factor) to the least number of $\delta$-balls needed to cover $P$. (This latter definition is the more common definition of covering number.) If $\cP\subset \cD_{\delta}$ is a set of dyadic cubes, we say that $\cP$ is a $(\delta, s, C)$-set if the union  of $\delta$-cubes in $\cP$ is a $(\delta, s, C)$-set in the sense of Definition \ref{defn:delta, s-set} and we use $|\cP|$ to denote the number of dyadic cubes in $\cP$ and $|\cP|_{\Delta}:=|\cD_{\Delta}(P)|$ (hence $|\cP|_{\delta}=|\cP|$).

\begin{definition}[Dyadic $\delta$-tubes]
    Let $\delta\in 2^{-\mathbb{N}}$. A \emph{dyadic $\delta$-tube} is a set of the form $T=\cup_{x\in p} \mathbf{D}(x)$, where $p\in \cD_{\delta}([-1,1]\times \mathbb{R}),$ and $\mathbf{D}$ is the \emph{point-line duality map}
\[
\mathbf{D}(a,b):=\{ (x,y)\in \mathbb{R}^2: y=ax+b\}\subset \mathbb{R}^2
\]
sending the point $(a,b)\in \mathbb{R}^2$ to a corresponding line in $\mathbb{R}^2$. Abusing the notation, we abbreviate $\mathbf{D}(p):=\cup_{x\in p} \mathbf{D}(x).$ The collection of all dyadic $\delta$-tubes is denoted 
\[
\cT^{\delta}:=\{ \mathbf{D}(p): p\in \mathcal{D}_{\delta}([-1,1]\times \mathbb{R})\}.
\]
A finite collection of dyadic $\delta$-tubes $\{\mathbf{D}(p)\}_{p\in \cP}$ is called a $(\delta, s, C)$-set if $\cP$ is a $(\delta, s, C)$-set. 

If $\delta<\Delta\in 2^{-\mathbb{N}}$ and $T\in \cT^{\delta}$, let $T^{\Delta}$ denote the dyadic tube $\bT\in \cT^{\Delta}$ containing $T$. 
\end{definition}
For a set of dyadic $\delta$-tubes $\cT$ and $\delta< \Delta\in 2^{-\mathbb{N}}$,  define the dyadic $\Delta$-covering number $|\cT|_{\Delta}:=|\cD_{\Delta}(\cT)|$ where $\cD_{\Delta}(\cT):=\{ \bT\in \cT^{\Delta}: \exists T\in \cT, \bT=T^{\Delta}\}$. When $\delta=\Delta$, sometimes we omit the subscript $\delta$ from $|T|_{\delta}$.

\begin{definition}[Slope set]
    The slope of a line $\ell=\mathbf{D}(a,b)$ is defined to be the number $a\in \mathbb{R}$, write $\dir(\ell):=a.$ 

    If $T=\mathbf{D}(p)$ is a dyadic $\delta$-tube, we define the slope $\dir(T) = \cup_{x\in p} \dir ( \mathbf{D}(x))$. Thus $\dir{T}$ is an interval of length $\sim \delta$. If $\mathcal{T}$ is a collection of dyadic $\delta$-tubes, we write $\dir(\cT):=\{\dir(T): T\in \cT\}$. If $\cT(p)$ is a collection of dyadic $\delta$-tubes through a fixed point $p$, we identify $\cT(p)$ with $\dir(\cT(p))$. 
\end{definition}

If $0<\delta <\Delta \leq 1$,  $\cT\subset \cT^{\delta}$ and $\bT\in \cT^{\Delta}$, write $\cT\cap \bT:=\{ T\in \cT: T\subset \bT\}$.

\subsection{Thickening and high-low method}

\begin{definition}[Incidences]\label{def: incidence}
   Let $\delta \leq \Delta\in 2^{-\mathbb{N}}$.  For a set of distinct weighted dyadic $\delta$-cubes $\cP$ with weight function $w : \cP \to \R^{\ge 0}$ and a set of (not necessarily distinct) dyadic $\Delta$-tubes $\T$, define the incidence count \[I_w (\cP, \T) := \sum_{p \in \cP} \sum_{T \in \T, p^{\Delta} \subset 6T} w(p).\]
For each dyadic $\Delta$-tube,  $6T:= T^{(6\Delta)}$, the $6\Delta$-neighborhood of $T$. 
    We usually assume $w \equiv 1$ or $w$ is explicitly defined beforehand, in which case we drop it from the notation.
\end{definition}
The following proposition is a slightly
 weaker version of \cite[Proposition 5.2.1]{bradshaw2022additive}. 
We include a  proof inspired by \cite{CPZ} here for completeness. 

\begin{prop}\label{prop:high low}
    Fix $ \delta\in 2^{-\mathbb{N}}$, let $\cP$ be a set of distinct weighted dyadic  $\delta$-cubes contained in $[0,1)^2$ with weight function  $w : \cP \to \R^{\ge 0}$ and $\cT$ be a set of distinct dyadic $\delta$-tubes. If $\delta^{-\eps/100} \le  S \le \delta^{-1}$, then
    \begin{equation*}
        I_w(\cP, \cT) \lesim_{\eps}  \left( S \delta^{-1} |\T| \sum_{p \in \cP} w(p)^2 \right)^{1/2} +
  S^{-1+\eps} I_w(\cP, \T^{S\delta}),
    \end{equation*}
    where $\cT^{S\delta} :=\{ T^{S\delta}: T\in \cT\}$ is counted with multiplicity.
\end{prop}
\begin{proof}
For each $T\in \cT$, let $\phi_T$ be the affine transformation that maps $T$ to  $[-1/2, 1/2]^2$.  Let $\psi$ be a smooth bump function $=1$ on $[-3, 3]^2$ and supported on $[-4, 4]^2$. Define $f_T= \psi\circ \phi_T$, then $f_T=1$ on $6T$ is supported on $8T$.  Then 
\[
I_w(\cP, \cT) \lesssim \delta^{-2}\int fg. 
\]
We are going to analyze $\int fg$ using Fourier transform. 

Since $|\widehat{\psi}(x)|\lesssim_N \frac{1}{(1+|x|^2)^N}$ for any $N >0$, the Fourier transform of $f_T$ is essentially supported on 
\[
T^*:=\{ \xi: |\xi\cdot (x-c_T)|\lesssim_{\eps} S^{\eps/2}, \quad \forall x\in T\} \text{ where } c_T \text{ is the center of } T, 
\]
in the sense that  $|\widehat{f}_T(\xi)|\leq S^{\eps^{-2}}\leq \delta^{1/\eps}$ for any $\xi \notin T^*$. 
Let $\psi_B$ be a smooth bump function $=1$ on $B=B(0, S^{-1+\eps/2}\delta^{-1})$ and supported on $2B$. 
Define the low frequency part of $f_T$ as 
\[
f_T^{l} = (\widehat{f}_T \psi_B)^{\vee}
\] and the high frequency part as $f_T^h= f_T-f_T^{l}$.  Then $|f_T^l|\leq S^{-1+\eps}$ on $N_{S\delta} T$ and $|f_{T}^l(x)|\lesssim_{\eps} \delta^{\eps^{-1}}$ for any $x\notin N_{S\delta}(T)$. In other words, $|f_T^l(x)|\lesssim_{\eps} S^{-1+\eps} \chi_{N_{S\delta}T} + \delta^{\eps^{-1}}$.  Define $f^l=\sum_{T\in \cT} f_T^l$ and $f^h=\sum_{T\in \cT} f_T^h$ and split the integral as 
\begin{equation}\label{eq: high-low}
\int fg \leq |\int f^h g|+|\int f^l g|.
\end{equation}

If the low frequency part dominates: $\int fg \lesssim |\int f^l g|$, then by Lemma~\ref{lem:thicken dyadic tube}, 
\begin{equation}\label{eq: low}
I_{w}(\cP, \cT) \lesssim \delta^{-2} \int fg \lesssim \delta^{-2} |\int f^l g|\lesssim_{\eps} S^{-1+\eps} I_{w}(\cP, \cT^{S\delta}) +\delta^{10}. 
\end{equation}
The last term $\delta^{10}$ is because the total number of distinct $\delta$-cubes and distinct $\delta$-tubes  is $\lesssim \delta^{-2}$, multiplying by $\delta^{\eps^{-1}}$ gives $\lesssim \delta^{10}$. 

It remains to consider when the high frequency part dominates:
\begin{equation} \label{eq: CauchySchwarz}
\int fg \lesssim |\int f^h g| \leq (\int |f^h|^2)^{1/2} (\int |g|^2)^{1/2}.
\end{equation} Since $T^*$ only depends on the direction of $T$, independent of its center $c_T$, we can organize $\cT=\sqcup_{\theta} \cT_{\theta}$ where  the disjoint union is over $\delta$-separated directions $\theta$ and $\cT_{\theta}$ consists of tubes in $\cT$ parallel to $\theta$ up to an error of size $\leq \delta$.  Write $f_{\theta}^h=\sum_{T\in \cT_{\theta}} f_T^h$ and  $f^h=\sum_{\theta} f_{\theta}^h$.  Then $\widehat{f_{\theta}^h}$ is essentially supported on $T^*\setminus B$.  Through each $\xi \in B^c$, at most $S$ many $\theta$ such that $\xi\in T^*$ for $T\in \cT_{\theta}$. Then
\begin{align*}
    \int |f^h|^2 &=\int |\widehat{f^h}|^2 \lesssim S\int \sum_{\theta} |\widehat{f_{\theta}^h}|^2\\
    & \lesssim S \sum_{\theta}\int |f_{\theta}|^2 \\
    &\lesssim S\sum_{\theta}\int \sum_{T\in \cT_{\theta}} |f_T|^2 \lesssim S \delta |\cT|. 
\end{align*}
Plugging in \eqref{eq: CauchySchwarz}, we get 
\begin{equation}\label{eq: high}
I_{w}(\cP, \cT) \lesssim \delta^{-2} \int fg \lesssim \delta^{-2} (S\delta |\cT|)^{1/2} (\delta^2 \sum_{p\in \cP} w(p)^2)^{1/2} \lesssim \big( S\delta^{-1} |\cT| \sum_{p\in \cP} w(p)^2\big)^{1/2}.  
\end{equation}
The conclusion follows from combining \eqref{eq: high-low}, \eqref{eq: low} and \eqref{eq: high}.

\end{proof}
\begin{remark} 
We often refers to the first term $( S \delta^{-1} |\T| \sum_{p \in \cP} w(p)^2 )^{1/2} $ as high-frequency term and the second term $S^{-1+\eps}I_w(\cP, \cT^{S\delta})$ as low-frequency term because they correspond to high-low frequences of $f$. 
Even $\cT^{S\delta}$ may not be a distinct set of dyadic $S\delta$-tubes, but $I_w(\cP, \cT^{S\delta})$ was defined in Definition~\ref{def: incidence}.  
\end{remark}


\begin{lemma}\label{lem:thicken dyadic tube}
    For each dyadic $\delta$-tube $T$ and $30\delta < \Delta \in 2^{-\mathbb{N}}$,
    \[
(6T)^{(\Delta)}\cap [-1, 2)^2 \subset 6(T^{\Delta}) .
    \]
    Recall $(6T)^{(\Delta)}$ means the $\Delta$-neighborhood of $6T$, i.e. the $\Delta$-neigborhood of $6\delta$-neighborhood of $T$,  and $T^{\Delta}$ is the dyadic $\Delta$-tube containing $T$. 
\end{lemma}


\begin{proof}
    Let $T = \bD(p)$ where $p = [a, a+\delta) \times [b, b+\delta)$ is a $\delta$-cube and let $\ell_{a,b}$ be the line $y = ax + b$.
    


    
    Let $(x, y) \in (6T)^{(\Delta)} \cap [-1, 2)^2$; then $d((x, y), \ell_{a,b}) \le \Delta+10\delta$. Since $a \in (-1, 1)$ and $\Delta> 30\delta$, this means $|y - (ax + b)| \le 2\Delta$.
    
    Let $p^\Delta = [a', a'+\Delta) \times [b', b'+\Delta)$, where $|a' - a| \le \Delta$ and $|b' - b| \le \Delta$ (this encodes the fact $p \in p^\Delta$). Then since $x \in [-1, 2)$,
    \begin{equation*}
        |y - (a'x + b')| \le |y - (ax + b)| + 2|a - a'| + |b - b'| \le 5\Delta.
    \end{equation*}
    Since $T^\Delta = \cD(p^\Delta)$, we get $(x, y) \in 6(T^\Delta)$, as desired.
\end{proof}

\section{Pigeonholing}
In this section, we pigeonhole a good configuration to have good properties.

\begin{definition}
    Fix $\delta\in 2^{-\mathbb{N}}, s\in [0,1], C>0, M\in \mathbb{N}$. We say that a pair $(\mathcal{P}, \mathcal{T})\subset \mathcal{D}_{\delta}\times \mathcal{T}^{\delta}$ is a $\delta$-configuration if every $p\in \cP$ is associated with a subset $\cT(p)\subset \cT$ such that $T\cap p\neq \emptyset$ for all $T\in \cT(p)$ (note that $\cT(p)$ does not need to be all tubes in $\cT$ passing through $p$).  We say that $(\cP,\cT)$ is a  $(\delta, s, C, M)$-nice configuration if for every $p\in \mathcal{P}$ there exists a $(\delta, s, C)$-set $\mathcal{T}(p)\subset \mathcal{T}$ with $|\mathcal{T}(p)|\sim M$ and such that $T\cap p\neq \emptyset$ for all $T\in \mathcal{T}(p)$. 
\end{definition}
It follows from the definition of $(\delta,s, C)$-set that $M\gtrsim C^{-1} \delta^{-s}$. 

\begin{definition}
    Let $\cP\subset \cD_{\delta}$. We say that $\cP'$ is a refinement of $\cP$ if $\cP'\subset \cP$ and $|\cP'|_{\delta}\approx_{\delta} |\cP|_{\delta}.$ Let $\Delta\in (\delta, 1)$, we say that $\cP'$ is a refinement of $\cP$ at resolution $\Delta$ if there is a refinement $\cP_{\Delta}'$ of $\cD_{\Delta}(\cP)$ and \[ \cP'=\bigcup_{\bp\in \cP_{\Delta}'}(\cP\cap \bp).\] 
    Note that $\cP'$  does not necessarily satisfy $|\cP'|_{\delta}\gtrapprox_{\delta}|\cP|_{\delta}$ unless we also have that $\cP$ is uniform at scale $\Delta$. We can likewise define a refinement of a set of dyadic $\delta$-tubes $\T$ at resolution $\Delta$.
\end{definition}
Here, the notation $\gtrapprox_\delta$ means $\gtrsim_N \log(1/\delta)^{C_N}$ for some number $N$ independent of $\delta$, and likewise for $\lessapprox_\delta, \approx_\delta$.

\begin{definition}\label{def: refinement}
    Let $(\cP_0, \cT_0)$ be a $(\delta, s, C_0, M_0)$-nice configuration. We say that a $\delta$-configuration $(\cP, \cT)$ is a refinement of $(\cP_0, \cT_0)$ if $\cP\subset \cP_0$, $|\cP|_{\delta}\approx_{\delta} |\cP_0|_{\delta}$, and for any $p\in \cP,$ there is $\cT(p)\subset \cT_0(p)\cap \cT$ with $\sum_{p\in \cP} |\cT(p)|_{\delta} \gtrapprox_{\delta} |\cP_0|_{\delta}\cdot M_0$. We say that a refinement $(\cP, \cT)$ of $(\cP_0, \cT_0)$  is a nice configuration refinement if in addition for each $p\in \cP$,  $|\cT(p)|\gtrapprox_{\delta} |\cT_0(p)|$.
    
    Let $\Delta\in (\delta, 1)$, we say that $(\cP, \cT)$ is a refinement of $(\cP_0,\cT_0)$ at resolution $\Delta$ if $\cP$ is a refinement of $\cP_0$ at resolution $\Delta$ and for each $p\in \cP$, $\bT\in \cD_{\Delta}(\cT(p))$, $\bT\cap \cT(p)=\bT\cap \cT_0(p)$. 
\end{definition}
Here $\cT$ might be much smaller than $\cT_0$. 
\begin{remark}\label{rmk: nice config refinement}
For any $(\delta, s, C_0, M_0)$-nice configuration $(\cP_0, \cT_0)$ and  any refinement $(\cP, \cT)$ of $(\cP_0, \cT_0)$, there is  a refinement $(\cP', \cT')$  of $(\cP, \cT)$ that is a nice configuration refinement of $(\cP_0, \cT_0)$. 
\end{remark}



\begin{definition}
    Let $N\geq 1, $ and let 
   \[
   \delta=\Delta_N<\Delta_{N-1}<\cdots<\Delta_1<\Delta_0=1
   \]
   be a sequence of dyadic scales. We say that a set $\cP\subset [0,1]^2$ is $\{\Delta_j\}_{j=1}^N$-uniform if there is a sequence $\{N_j\}_{j=1}^N$ with $N_j\in 2^{\mathbb{N}}$ and $|\cP\cap Q|_{\Delta_j}\in [ N_j/2, N_j)$ for all $j\in \{1, \dots, N\}$ and all $Q\in \cD_{\Delta_{j-1}}(\cP)$. 
\end{definition}

Recall Lemma 2.15 from \cite{orponen2023projections}.
\begin{lemma}\label{lem: pigeonhole}
    Let $\cP\subset [0,1)^d, N, T\in \mathbb{N}$ and $\delta=2^{-NT}$. Let also $\Delta_j:=2^{-jT}$ for $0\leq j\leq N$. Then, there is a $\{\Delta_j\}_{j=1}^N$-uniform set $\cP'\subset \cP$ such that 
    \[
    |\cP'|_{\delta}\geq (2T)^{-N} |\cP|_{\delta}. 
    \]
    In particular, if $\epsilon>0$ and $T^{-1}\log(2T)\leq \epsilon$, then $|\cP'|_{\delta}\geq \delta^{\epsilon}|\cP|_{\delta}.$
\end{lemma}

\section{Main Theorem on Furstenberg sets}

In this section, we will state our main theorem on Furstenberg sets and prove the special case $\cP$ is a $(\delta, 2-s, C)$-set, see Proposition~\ref{prop:s+t=2}. Then, we will prove a key special case (the semi well-spaced case) where $\cD_\Delta (\cP)$ looks like a $(\Delta, 2-s, C)$-set at  a high scale $\Delta$ and $\cP$ looks like a $(\delta, s, C)$-Katz-Tao set at low scales between $\delta$ and $\Delta$, see Proposition~\ref{prop: key}. Finally, we will combine the semi-well-spaced case (Proposition~\ref{prop: key})  and  almost AD-regular case (the recent breakthrough of \cite{orponen2023projections})  using a carefully designed scale decomposition (Lemma~\ref{lem:dictionary}) and  a now-standard induction on scales argument due to \cite{orponen2021hausdorff} to finish off the proof of Theorem~\ref{Thm: main}.

In order to prove the Furstenberg set conjecture, it suffices to prove the following discretized version (see \cite[Section 3.1]{orponen2021hausdorff} or \cite[Lemma 3.3]{hera2021improved}).
\begin{theorem}\label{Thm: main}
   For every $\eps > 0$, there exists $\eta = \eta(\eps,s,t) > 0$ such that the following holds for any $ \delta < \delta_0(s, t, \epsilon)$. Let $(\cP, \cT)$ be a $(\delta, s, \delta^{-\eta}, M)$-nice configuration with $s\in (0, 1]$, $\cP$ is a $(\delta, t, \delta^{-\eta})$-set, $t\in (0, 2]$. Then
    \[
    |\cT|_{\delta} \gtrsim_\eps \delta^{-\min\{t, \frac{s+t}{2},1\}+\eps} M.
    \]
\end{theorem}

Theorem~\ref{Thm: main} will be proved using induction on scales.  To do so, we need the following definition and pigeonholing lemma (Proposition~\ref{prop: refinement}) that will be used in the proof of Proposition~\ref{prop:s+t=2} and the proof of Proposition~\ref{prop: key}. 

\begin{definition}\label{def: cover}
   Fix $\delta\leq  \Delta \in 2^{-\mathbb{N}}$, $s\in [0,1], \,  C_{\delta}>0,  \, M_{\delta}\in \mathbb{N}$.   We say that a $\Delta$-configuration $(\cQ, \cT_{\Delta})$ covers a $(\delta, s, C_{\delta}, M_{\delta})$-nice configuration $(\cP, \cT)$ if 
   \begin{equation}\label{eq: cover}
       \sum_{\bp\in \cQ} \,\,  \sum_{p\in \bp\cap \cP}  \,\, \sum_{\bT\in \cT_{\Delta}(\bp)}|\cT_p\cap \bT|_{\delta} \gtrapprox_{\delta} |\cP|_{\delta} \cdot M_{\delta}. 
   \end{equation}
\end{definition}

\begin{prop}\label{prop: refinement}
    Fix $\delta\in 2^{-\mathbb{N}}, s\in (0,1], C>0, M\in \mathbb{N}$ and let $\Delta\in (\delta, 1)$. Let $(\cP, \cT)$ be a $(\delta, s, C, M)$-nice configuration, then there exists a  refinement $(\cP_1, \cT_1)$ of $(\cP, \cT)$  such that $\cP_1$ is a $\{1, \Delta, \delta\}$-uniform set and 
    \begin{enumerate}
        \item\label{it: thickening} $(\cP_1^{\Delta}, \cT_1^{\Delta})$, $ \cP_1^{\Delta}:=\cD_{\Delta}(\cP_1)$ and for any $\bp\in \cP_1^{\Delta}$,  $\cT_{1}( \bp)^{\Delta}:=\cD_{\Delta} (\underset{p\in \cP_1\cap \bp}{\cup} \cT_{1}(p))$, is a $(\Delta, s, C_{\Delta}, M_{\Delta})$-nice configuration with $C_{\Delta} \lessapprox_{\delta} C$ and  $M_{\Delta}\in \mathbb{N}$. 
        \item \label{it: refinement}     any refinement $(\cQ, \cT_{\Delta})$ of  $(\cP_1^{\Delta}, \cT_1^{\Delta})$ covers $(\cP, \cT)$. 
    \end{enumerate}
\end{prop}
The proof of Proposition~\ref{prop: refinement} is a slight modification of \cite[Proposition 4.1]{orponen2021hausdorff}, which we include for completeness. 
\begin{proof}
By Lemma~\ref{lem: pigeonhole}, for every $p\in \cP$, there exists a $\{1, \Delta, \delta\}$-uniform subset $\cT_1(p)\subset \cT(p)$ with $|\cT_1(p)|\approx_{\delta} M$. Let $m(p)=|\cT_1(p)|_{\Delta}$. Then for each $\bT\cap  \cT_1(p) \neq \emptyset$, $|\bT\cap \cT_1(p)|_{\delta} \approx_{\delta} \frac{M}{m(p)}$.
By pigeonholing and Lemma~\ref{lem: pigeonhole} again, 
we can find a dyadic number $m$ and a $\{1, \Delta, \delta\}$-uniform subset   $\cP_1\subset \cP$, $|\cP_1|_{\delta} \gtrapprox_{\delta} |\cP|_{\delta}$    and  for every  $p\in \cP_1$, $m\sim m(p)$.     

For each $\bp \in \cD_{\Delta}(\cP_1)$, there exists a dyadic number $X(\bp)$ and  a subset $\cT_1(\bp)\subset \cT^{\Delta}$ satisfying  for each $\bT\in \cT_1(\bp)$, 
\[
|\{ p\in \cP_1\cap \bp: \bT\in \cD_{\Delta}(\cT_1(p))\}|\sim X(\bp)
\]
and 
\begin{equation}\label{eq: X}
X(\bp) \cdot |\cT_1(\bp)|_{\Delta} \approx_{\Delta}|\cP_1\cap \bp|_{\delta} \cdot  m
\end{equation}

By pigeonholing, there exists a common dyadic number $X$ and a refinement of $\cP_1$ 
 $\cP_1$ at resolution $\Delta$, which we still denote $\cP_1$, such that  for each $\bp\in \cD_{\Delta}(\cP_1)$, $X(\bp)=X$.  We claim that  for each $\bp\in \cD_{\Delta}(\cP_1)$, $\cT_1(\bp)$ is $(\Delta, s, C_{\Delta})$-set for $C_{\Delta} \approx_{\delta} C$ and
 \begin{equation}\label{eq: cT1bp} |\cT_1(\bp)|_{\Delta} \approx_{\delta} M_{\Delta}:=  \frac{ |\cP_1\cap \bp|_{\delta} \cdot  m }{  X} \end{equation} 
 (since $\cP_1$ is $\{1, \Delta, \delta\}$-uniform, it does not matter which $\bp\in \cD_{\Delta}(\cP_1)$ we take in the definition of $M_{\Delta}$).    To see this,  for any $\rho \in (\Delta, 1)$, and any $\bT_{\rho}\in \cT^{\rho}$,  double count the number of pairs $(p, T)\in (\cP_1\cap \bp )\times \cT$ such that $ T\in \cT_1(p)\cap \bT_{\rho}$, we have 
\begin{align*}
|\cT_1(\bp)\cap \bT_{\rho}|_{\Delta}   &  \approx_{\delta}  \frac{ |\cP_1\cap \bp|_{\delta} \cdot |\cT_1(p)\cap \bT_{\rho}|_{\delta} \cdot m} { M\cdot  X} \\
&\lessapprox_{\delta} C \rho^{s} \frac{ |\cP_1\cap \bp|_{\delta} \cdot m }{X} \approx_{\delta} C\rho^s \cdot |\cT_1(\bp)|_{\Delta}. 
\end{align*}
The first line implies the second line because $\cT_1(p) \subset \cT(p)$ and $\cT(p)$ is a $(\delta, s,  C)$-set.

For any $p\in \cP_1$, replace   $\cT_1 (p)$ by $\{T\in \cT_1(p):  T^{\Delta} \in \cT_1(p^{\Delta})\}$ (this is a refinement at resolution $\Delta$).  Then \begin{align*}
    \sum_{p\in \cP_1} |\cT_1(p)|_{\delta} & =\sum_{\bp\in \cD_{\Delta}(\cP_1)} \sum_{p\in \cP_1\cap \bp} \sum_{\bT\in  \cT_1(\bp)} |\cT_1(p)\cap \bT|_{\delta} \\
    & =\sum_{\bp\in \cD_{\Delta}(\cP_1)}  \sum_{\bT\in  \cT_1(\bp)} \sum_{p\in \cP_1\cap \bp} |\cT_1(p)\cap \bT|_{\delta} \\
    & \gtrapprox_{\delta} \sum_{\bp\in \cD_{\Delta}(\cP_1)} |\cT_1(\bp)|_{\Delta} \cdot  \frac{M}{m} X   \\
    &\gtrapprox_{\delta} M \cdot |\cP_1|_{\delta} \gtrapprox_{\delta} M \cdot |\cP|_{\delta}. 
\end{align*} 
The third inequality is because $\sum_{p\in \cP_1\cap \bp} |\cT_1(p)\cap \bT|_{\delta}\sim X\cdot \frac{M}{m}$ and the fourth inequality is due to \eqref{eq: cT1bp} and uniformity of $\cP_1$. 

By definition of $\cT_1(p)$, we have $\cD_{\Delta}(\cup_{p\in \cP_1\cap \bp}\cT_1(p)) = \cT_1(\bp)$ for any $\bp\in \cP_1^{\Delta} : =\cD_{\Delta}(\cP_1)$, and so \eqref{it: thickening} holds. To verify \eqref{it: refinement}, for any refinement $(\cQ_{\Delta}, \cT_{\Delta})$ of  $(\cP_1^{\Delta}, \cT_1^{\Delta})$, 
\begin{align*}
    \sum_{\bp\in \cQ_{\Delta}}  \, \sum_{p\in \cP\cap \bp}  \, \sum_{\bT\in \cT_{\Delta}(\bp)} |\cT(p)\cap \bT|_{\delta} & \gtrapprox_{\delta}     \sum_{\bp\in \cQ_{\Delta}} \, \sum_{p\in \cP_1\cap \bp} \,  \sum_{\bT\in \cT_{\Delta} (\bp)} |\cT_1(p)\cap \bT|_{\delta} \\
     & =   \sum_{\bp\in \cQ_{\Delta}} \,  \sum_{\bT\in \cT_{\Delta} (\bp)} \, \sum_{p\in \cP_1\cap \bp}  |\cT_1(p)\cap \bT|_{\delta} \\
    &\gtrapprox_{\delta} \sum_{\bp\in \cQ} |\cT_{\Delta} (\bp)|_{\Delta} \cdot  \frac{M}{m} X    \\
    &\gtrapprox_{\delta} \sum_{\bp\in \cD_{\Delta}(\cP_1)} |\cT_1(\bp)|_{\Delta} \cdot  \frac{M}{m} X  \gtrapprox_{\delta} M\cdot |\cP|_{\delta}. 
\end{align*}
In the second to last inequality, we used the assumption that $(\cQ_{\Delta}, \cT_{\Delta})$ is a refinement of $(\cP_1^{\Delta}, \cT_1^{\Delta})$.

\end{proof}

\subsection{The case $s + t \geq 2$}

Then the conclusion of Theorem \ref{Thm: main} becomes $|\T|_\delta \gesim_\eps \delta^{-1+\eps} M$. By double counting the number of ``restricted incidences'' $|\{ (p, T) \in \cP \times \cT : T \in \T(p) \}|$, we can equivalently say the average tube in $\T$ belongs to $\T(p)$ for $\lesim \delta^{1-\eps} |\cP|_\delta$ many $\delta$-cubes $p \in \cP$. For our purposes, we need the following stronger statement: every tube in $\T$ contains $\lesim \delta^{1-\eps} |\cP|_\delta$ many $\delta$-cubes in $\cP$. Furthermore, we need this estimate to be true at all scales $w \in (\delta, 1)$, in the quantitative form presented below. These strong properties may not be true of the original configuration $(\cP, \cT)$, but will be true after passing to a refinement.

\begin{prop}\label{prop:s+t=2}
    Fix $ \Delta < \Delta_0$, $n \in \N$, and let $\delta = \Delta^n$. (Here $\Delta_0 \in ( 0, 1)$ is an absolute constant.) Let $(\cP_0, \cT_0)$ be a $(\delta, s, C_0, M_0)$-nice configuration with $\cP$ a $(\delta, t, C_0)$-set, $s+t\geq 2$. Then there exists a refinement $(\cP, \cT)$ of $(\cP_0, \cT_0)$ such that for any $1 \le k \le n$ and any $T\in \cT$, $w = \Delta^k$, $\eps\in (0, 1/n]$, 
    \begin{equation}\label{eqn:cor s+t=2}
        |6(T^w) \cap \cD_w (\cP)| \lessapprox_{\delta, \eps} C_0^2 \Delta^{-1} \cdot |\cP|_w  \cdot w^{1-\eps}.
    \end{equation}
    Here $|6(T^w)\cap \cD_w(\cP)|$ means the number of $\bp\in \cD_w(\cP)$ contained in $6(T^w)$, the $6w$-neighborhood of $T^w$. 
    The notation $\lessapprox_{\delta, \eps}$ means  $\leq C_\eps |\log \delta|^{C}$, where the implicit constants $C$ and $C_{\eps}$ may depend on $\Delta_0$ and $n$ but is independent of $\Delta$, $C_0$, or $M_0$, in addition, $C_{\eps}$ depends on $\eps$. 
\end{prop}

\begin{remark}
    This proposition will be used in proving Proposition \ref{prop: key} with the parameter $n = \eps^{-1}$. Thus, $\Delta^{-1}$ should be viewed as a small power of $\delta^{-1}$.
\end{remark}

\begin{proof}
    We induct on $n$. Base case $n = 1$ is the trivial fact $|T \cap \cD_{\Delta} (\cP)| \le \Delta^{-1}$ (which is why we have a $\Delta^{-1}$ term on the RHS of \eqref{eqn:cor s+t=2}).

    For the inductive step, assume true for $n-1$, and we will now consider $n$. We first refine $\cP_0$ to be uniform at scales $\{ \Delta^i \}_{i=0}^n$ (and call the refinement $\cP_0$). Next, we apply Proposition~\ref{prop: refinement} with $\Delta^{n-1}$ in the place of $\Delta$ to obtain a refinement $(\cP_1, \cT_1)$ of $(\cP_0, \cT_0)$ and the corresponding $(\Delta^{n-1}, s, C_{\Delta^{n-1}}, M_{\Delta^{n-1}})$-nice configuration $(\cP_1^{\Delta^{n-1}}, \cT_1^{\Delta^{n-1}})$. By the induction hypothesis, we can find a refinement $(\cQ_{\Delta^{n-1}}, \cT_{\Delta^{n-1}})$ of $(\cP_1^{\Delta^{n-1}}, \cT_1^{\Delta^{n-1}})$ such that for any $1\leq k\leq n-1$ and any $\bT\in \cT_{\Delta^{n-1}}$, $w=\Delta^k$, 
        \begin{equation}\label{eqn:cor s+t=2 ind hyp}
        |6(\bT^w)\cap \cD_w (\cQ_{\Delta^{n-1}})| \lessapprox_{\delta,\eps} C_0^2 \Delta^{-1} \cdot |\cQ_{\Delta^{n-1}}|_w \cdot w^{1-\eps}.
    \end{equation}

Let $\cP_2 := \bigcup_{\bp \in \cQ_{\Delta^{n-1}}} (\cP_1 \cap \bp)$ and for each $p \in \cP_2$, let $\cT_2 (p) := \{ T \in \cT_1 (p) : T^{\Delta^{n-1}} \in \cT_{\Delta^{n-1}} \}$. Then by conclusion~\eqref{it: refinement} of Proposition~\ref{prop: refinement}, 
\begin{equation}
    \label{eqn:many tubes left}
    \sum_{p\in \cP_2}|\cT_2(p)|_{\delta}\gtrapprox_{\delta} |\cP_1|_{\delta}\cdot M_0
\end{equation}


Roughly speaking, \eqref{eqn:many tubes left} tells us that the configuration $(\cP_2, \cT_2)$ ``preserves many of the restricted incidences'' of $(\cP_1, \cT_1)$, where the set of restricted incidences is $\{ (p, T) \in \cP_1 \times \cT_1 : T \in \cT_1 (p) \}$.  By pigeonholing, there exists $r$ such that defining $\cT= \{ T \in \cT_2  : |6T \cap \cP_2|\sim r \} $ and for each $p\in \cP_2$, $\cT(p): =\cT_2(p)\cap \cT$, 
\[
\sum_{p\in \cP_2} |\cT(p)|_{\delta}\gtrapprox_{\delta} |\cP_1|_{\delta} \cdot M_0. 
\]


After this procedure, we have two useful properties for each $T \in \cT$:
\begin{itemize}
    \item $|6T \cap \cP_2|\sim r$ where $r \in [1, \delta^{-1}]$;

    \item $T^{\Delta^{n-1}} \in \T_{\Delta^{n-1}}$ (because $T \in \T_2 (p)$ for some $p \in \cP_2$), so \eqref{eqn:cor s+t=2 ind hyp} applies to $T^{\Delta^{n-1}}$.
\end{itemize}

Now, we check \eqref{eqn:cor s+t=2}. Fix $T \in \T$ and $w = \Delta^k$, $1 \le k \le n$. If $1 \le k \le n-1$, then \eqref{eqn:cor s+t=2 ind hyp} applied to $T^{\Delta^{n-1}}$ gives
\begin{equation*}
    |6(T^w) \cap \cD_w (\cP_2)| \lessapprox_{\delta,\eps} C_0^2 \Delta^{-1} \cdot |\cQ_{\Delta^{n-1}}|_w \cdot w^{1-\eps} \leapp_\Delta C_0^2 \Delta^{-1} \cdot |\cP_2|_w \cdot w^{1-\eps}.
\end{equation*}
Now, we check \eqref{eqn:cor s+t=2} for $k = n$. For this, it suffices to obtain an upper bound on $r$. To estimate $r$, we use the high-low method (Proposition \ref{prop:high low}) to study the number of (unrestricted) incidences between $\cT$ and $\cP_2$ defined in Definition~\ref{def: incidence} with $w(p)=1$.  Let $S = \Delta^{-1} = \delta^{-1/n}$  in Proposition \ref{prop:high low}; then
\begin{equation}\label{eqn:prop 4.2 high low}
    I(\cP_2, \T) \lesim_{ \eps} \Delta^{-1/2} |\cP_2|^{1/2} (|\cT|_{\delta} \delta^{-1})^{1/2} + \Delta^{1-\eps} I(\cP_2, \T^{\Delta^{n-1}}).
\end{equation}
If the first term on the right-hand side dominates, then the facts $I(\cP_2, \T) \geapp_\Delta \max(|\cP_2|_{\delta} M_0, r |\cT|_{\delta})$ imply that
\[
(|\cP_2|_{\delta} M)^{1/2} (r |\cT|_{\delta})^{1/2} \leapp_\Delta I(\cP_2, \T) \lesim_{ \eps} \Delta^{-1/2} |\cP_2|^{1/2} (|\cT|_{\delta} \delta^{-1})^{1/2}.
\]
Since $|\cP_2|_{\delta} \gtrapprox_{\delta} |\cP_0|_{\delta} \ge C_0^{-1} \delta^{-t}$, $M_0 \ge C_0^{-1} \delta^{-s}$, and $s + t \ge 2$, this implies that 
\[
r \leq \Delta^{-1} \delta^{-1} M_0^{-1}\leq C_0^2 \Delta^{-1} \cdot |\cP|_{\delta} \cdot  \delta.
\]

If the second term on the RHS of \eqref{eqn:prop 4.2 high low} dominates, let $\bp\in \cD_{\Delta^{n-1}}(\cP_2)$, then 
\begin{align*}
    r \cdot |\cT|_{\delta} &\lesim_{\eps} \Delta^{1-\eps}  |\{ (p, T)\in \cP_2 \times \cT: p^{\Delta^{n-1}} \subset 6(T^{\Delta^{n-1}})\}| \\
    &\lesim_{\eps} \Delta^{1-\eps}  |\{ (\bp, T)\in \cD_{\Delta^{n-1}}(\cP_2)\times \cT: \bp\subset 6(T^{\Delta^{n-1}})\}| \cdot \sup_{\bp \in \cD_{\Delta^{n-1}} (\cP_2)}|\cP_2 \cap \bp|_\delta \\
    &\lessapprox_{\delta,\eps}  \Delta^{1-\eps}   |\T|_{\delta} \cdot C_0^2 \cdot |\cP_2|_{\Delta^{n-1}} \cdot \Delta^{-1}   \cdot  \Delta^{(n-1)(1-\eps)} \cdot \frac{|\cP_2|_\delta}{|\cP_2|_{\Delta^{n-1}}} \\
    &\lessapprox_{\delta,\eps} |\cP_2|_{\delta}\cdot  |\cT|_{\delta} \cdot\Delta^{-1} \cdot  C_0^2 \delta^{1-\eps}.
\end{align*}
The third step follows from the inductive hypothesis \eqref{eqn:cor s+t=2 ind hyp} applied to $\bT = T^{\Delta^{n-1}} \in \cD_{\Delta^{n-1}} (\T_2) \subset \T_{\Delta^{n-1}} (\bp)$ with $k = n-1$ (using $\cD_{\Delta^{n-1}} (\cP_2) = \cQ_{\Delta^{n-1}}$), and uniformity of $\cP_2$ at scale $\Delta^{n-1}$.
Therefore, 
\[
r \lessapprox_{\delta,\eps} C_0^2 \Delta^{-1} |\cP_2|_{\delta} \cdot \delta^{1-\eps}.
\]
This holds in either case with $\cP=\cP_2$ and thus we have completed the proof of the inductive step.

\end{proof}
\subsection{When $\cP$ is $(2-s)$-dimensional at high scales and $s$-dimensional at low scales}

\begin{prop}\label{prop: key}
    Let $\delta, \Delta \in 2^{-\mathbb{N}}, \delta\leq \Delta$, $\eps > 0$, and $(\cP, \cT)$ be a $(\delta, s, \delta^{-\eps^2}, M)$-nice configuration with $s\in (0, 1]$. Suppose that
    $\cP \subset \cD_\delta$ satisfies the spacing conditions
    \begin{gather}
        |\cP \cap B(x, \rho)|_{\delta} \le \delta^{-\eps^2} \cdot \rho^{2-s} \cdot |\cP|_{\delta}, \qquad \Delta < \rho < 1, \\
        |\cP \cap B(x, \rho)|_{\delta} \le \delta^{-\eps^2} \cdot (\rho/\delta)^s, \qquad \delta < \rho < \Delta. \label{eqn:spacing condition 2}
    \end{gather}
    Then $|\cT|_{\delta} \gtrsim_\eps \delta^{35\eps} \cdot \min(M|\cP|_{\delta}, M^{3/2} |\cP|_{\delta}^{1/2}, \delta^{-1} M)$.
\end{prop}

For future reference, we repackage it into a special case of Theorem \ref{Thm: main}. The corollary follows from Proposition \ref{prop: key} because we can just take $\Delta = (|\cP|_\delta 
 \delta^s)^{-1/(2s-2)} \in [\delta, 1]$.
\begin{cor}\label{cor: key}
    For every $\eps > 0$ and $\eta = \frac{\eps^2}{1225}$, the following holds for any $0< \delta < \delta_0(s, t, \epsilon)$. Let $(\cP, \cT)$ be a $(\delta, s, \delta^{-\eta}, M)$-nice configuration with $s\in (0, 1]$, and $|\cP| \sim \delta^{-t}$, $t \in [s, 2-s]$.
    Suppose $\cP$ satisfies the stronger spacing condition
    \begin{equation*}
        |\cP \cap Q|_\delta \lesim \delta^{-\eta} \cdot \max(\rho^{2-s} |\cP|_\delta, (\rho/\delta)^s) \text{ for any } Q \in \cD_\rho (\cP), \delta \le \rho \le 1.
    \end{equation*}
   Then
    \[
    |\cT|_{\delta} \gtrsim_\eps \delta^{-\frac{s+t}{2}+\eps} M.
    \]
\end{cor}

The key lemma to prove Proposition \ref{prop: key} is the following generalization of \cite[Theorem 1.1]{guth2019incidence}. For technical reasons, we find it convenient to work with multi-sets of dyadic $\delta$-cubes (with possible repetitions). 

 For any dyadic $\delta$-tube $T$ and $b\in \mathbb{N}$,  let $N_{w, b} (T)$ be the number of $Q \in \cD_w (\cP)$ (counted with multiplicity) such that $|4T \cap Q \cap \cP| \ge b$. 
\begin{lemma}\label{lem:r-rich}
    Let $\delta < \frac{\Delta}{32} \in 2^{-\mathbb{N}}$, $\eps \in \frac{1}{\N}$, and let $\cP$ be a multi-set of dyadic $\delta$-cubes such that for all $1 \le k \le \eps^{-1}$,  each $Q \in \cD_{\Delta^{k\eps}} (\cP)$ contains about the same number of cubes in $\cP$ (with multiplicity), and for any $Q_1\neq  Q_2 \in \cD_\Delta (\cP)$, $\text{dist} (Q_1, Q_2)\geq \Delta$.
    For $a \ge 2$ and $ab > \delta^{1-2\eps} |\cP|$, let  $\T_{a,b}$ be a set of  distinct  dyadic $\delta$-tubes satisfying 
    \begin{enumerate}
        \item \label{it: good} $|6T^w \cap \cD_w (\cP)| \le \Delta^{-\eps} \cdot w |\cD_w (\cP)|$ for all $w \in \{ \Delta^\eps, \Delta^{2\eps}, \cdots, \Delta \}$,
        \item $N_{\Delta, b} (T) \ge a$.
    \end{enumerate}
    Then 
    \begin{equation}
        |\T_{a,b}| \lesim_{\eps} \frac{|\cP|^2}{a^3 b^2} \delta^{-5\eps}.
    \end{equation}
\end{lemma}


\begin{remark}
    If we superimpose $k$ copies of $\cP$ and increase $b$ by a factor of $k$, then Lemma~\ref{lem:r-rich} does not change.
\end{remark}

We give some clarifying examples of this lemma.
\begin{remark}
   Lemma~\ref{lem:r-rich} holds equally well for general $\delta$-balls in place of dyadic $\delta$-cubes. To compare Lemma~\ref{lem:r-rich} with the well-spaced case proved in \cite[Theorem 3]{fu2022incidence} (the dual of \cite[Theorem 1.1]{guth2019incidence}), 
    if $\cP$ is a set of well-spaced $\delta$-balls, i.e. $\cD_\Delta (\cP)$ is the set of all $\Delta$-separated dyadic $\Delta$-squares in $[0, 1]^2$ and each $Q \in \cD_{\Delta} (\cP)$ contains exactly $1$ $\delta$-ball in $\cP$, then the set of $r$-rich tubes is contained in $\T_{r,1}$. Furthermore, for any $\eps > 0$, the condition $N_w (T) \le C\Delta^{-\eps} \cdot w|\cD_w (\cP)|$ is true: the LHS is $\le w^{-1}$ while $|\cD_w (\cP)| = w^{-2}$. Hence, we get $|\T_r| \lesim \frac{|\cP|^2}{r^3} \delta^{-C \eps}$, which agrees with \cite[Theorem 3]{fu2022incidence} (the dual of \cite[Theorem 1.1]{guth2019incidence}) in the case $X = W$. (Curiously, Lemma \ref{lem:r-rich} doesn't seem to imply all of \cite[Theorem 3]{fu2022incidence}.)
\end{remark}

Before proving Lemma~\ref{lem:r-rich}, we see how it implies Proposition~\ref{prop: key}. A key ingredient is the following lemma, which  uses the spacing condition \eqref{eqn:spacing condition 2} for $\cP$.

\begin{lemma}\label{lem:elementary incidence}
    Fix $\delta < \Delta \in 2^{-\mathbb{N}}$. Let $\cP$ be a set of distinct dyadic $\delta$-cubes in $[0, \Delta]^2$ such that $|\cP \cap B(x, w)| \le K \cdot (\frac{w}{\delta})^s$ for all $\delta \le w \le \Delta$ and $w$-balls $B(x, w)$. For each $p \in \cP$, let $\T(p)$ be a $(\delta, s, K)$-set of dyadic tubes such that $|\T(p)| \sim  M$. Let $\T = \cup \T(p)$. Then there exists a subset $\cP' \subset \cP$ with $|\cP'| \ge \frac{1}{2} |\cP|$ and for each $p \in \cP$, a subset $\T'(p) \subset \T(p)$ with $|\T'(p)| \ge \frac{1}{2} |\T(p)|$ such that for each $T \in \T'$, we have $|\{ p \in \cP' : T \in \T'(p) \}| \le C_1 K^2 |\log \delta|$, for some universal constant $C_1 > 0$.
\end{lemma}
Lemma~\ref{lem:elementary incidence} also reproves the $t\leq s$ case of Furstenberg sets conjecture, i.e. the work of \cite{lutz2020bounding} and \cite{hera2019hausdorff}.

\begin{proof}
    Let $\Tbad \subset \T$ be the set of tubes with $|\{ p \in \cP : T \in \T(p) \}| \ge C_1 K^2|\log \delta|$. By Markov's inequality, we get $|\Tbad| \lesim \frac{M|P|}{C_1 K^2 |\log \delta|}$. This is the only property of $\Tbad$ that will be used in the proof.

    Suppose the lemma were false; then there exists a subset $|\cP'| \ge \frac{1}{2} |\cP|$ and a subset $\T'(p) \subset \T(p) \cap \Tbad$ for each $p \in \cP'$ with $|\T'(p)| \ge \frac{1}{2} |\T(p)|$. Note that $\T'(p)$ is still a $(\delta, s, 2K)$-set. We now find a lower bound for $\T' = \cup_{p \in \cP} \T'(p)$ via an elementary energy argument. Let $J$ be the number of triples $(p_1, p_2, T) \in (\cP')^2 \times \cT'$ such that $T \in \T'(p_1) \cap \T'(p_2)$. Then
    \begin{equation*}
        J \lesim \sum_{p \in \cP'} \sum_{\substack{w \text{ dyadic},\\\delta < w < \Delta}} (KM(\delta/w)^s) \cdot K \cdot (w/\delta)^s + |\cP| M \le M(K^2 |\log \delta| + 1) |\cP|.
    \end{equation*}
    Here, there are $\leq K \cdot (w/\delta)^s$ many $\delta$-cubes $q \in \cP'$ at distance $\sim w$ from $p$, and there are $\lesim KM(\delta/w)^s$ many $\delta$-tubes in $\cT(p)\cap \cT(q)$ with $d(p, q) \sim w$.
    
    Thus, by Cauchy-Schwarz, \[\frac{(M |\cP|)^2}{|\T'|} \lesim MK^2 |\log \delta|\cdot  |\cP|,\]
    so $|\T'| \gesim (K^2 |\log \delta|)^{-1} M |\cP|$. Since $\T' \subset \Tbad$, this contradicts our upper bound $|\Tbad| \lesim \frac{M|P|}{C_1 K^2 |\log \delta|}$ if $C_1$ is chosen sufficiently large.
\end{proof}

\begin{proof}[Proof of Proposition~\ref{prop: key} given Lemma~\ref{lem:r-rich}]
    Pick $\delta_0 = \delta_0 (\eps)$, to be chosen sufficiently small later. First, if $\delta > \delta_0$, then the trivial bound $|\T|_\delta \ge M$ would suffice when the implicit constant in $\gtrapprox_{\eps}$ is sufficiently large depending on $\delta_0(\eps)$. Hence, assume $\delta < \delta_0$.

    First, if $\delta^{8\eps} < \Delta$, then $\cP$ is a $(\delta, s, \delta^{-17\eps})$ Katz-Tao set, since for $\rho < \Delta$ we have \eqref{eqn:spacing condition 2}, and for $\rho \ge \Delta$, we have
    \begin{equation*}
        |\cP \cap B(x, \rho)| \lesim \left( \frac{\rho}{\Delta} \right)^2 \sup_x |\cP \cap B(x, \Delta)| \lesim \Delta^{-2} \cdot \delta^{-\eps^2} \left( \frac{\Delta}{\delta} \right)^s \le \delta^{-17\eps} \cdot \left( \frac{\rho}{\delta} \right)^s.
    \end{equation*}
    Hence, by Lemma \ref{lem:elementary incidence}, there exists a subset $\cP'\subset\cP$ with $|\cP'|\geq \frac{1}{2}|\cP|$ and for each $p\in \cP$ a subset $\cT'(p)\subset \cT(p)$ with $|\cT'(p)|\geq \frac{1}{2}|\cT(p)|\sim M$ such that for each $T\in \cT'$, 
    \[
    |\{p\in \cP': T\in \cT'(p)\}|\leq C_1 \delta^{-34\eps} |\log \delta|.
    \]
    Therefore, 
    \begin{equation*}
        |\T| \cdot \sup_{T \in \T'} |\{ p \in \cP' : T \in \T'(p) \}| \ge |\cP'| \sup_{p \in \cP'} |\T'(p)| \gesim |\cP| M \implies |\T| \gtrsim_{\eps} |\cP| M \delta^{35\eps}.
    \end{equation*}
    Thus, assume $\delta^{8\eps} > \Delta$. In particular, $\Delta < \delta_0^{8\eps}$.

    Next, if $\delta > \frac{\Delta}{100}$, then $\cP$ is a $(\delta, 2-s, \delta^{-\eps})$-set for analogous reasons, and we may conclude $|\T| \gesim_\eps \delta^{-1+4\eps} M$ by  Proposition~\ref{prop:s+t=2}  (see also \cite[Theorem 5.28]{orponen2023projections} and  \cite[Theorem 5.2]{fu2021incidence}). Thus, assume $\delta < \frac{\Delta}{100}$.

    The main idea is to apply Lemma \ref{lem:r-rich} to the configuration $(\cP, \cT)$ to get a strong incidence bound. The impatient reader can skip to the end of the proof for the calculation.

    But to satisfy the conditions of Lemma \ref{lem:r-rich}, we need to suitably refine $(\cP, \cT)$. 
    First, apply Lemma~\ref{lem: pigeonhole} to $\cP$ such that the resulting set, still denote as $\cP$ is $\{\Delta_j\}_{j=1}^{\eps^{-1}}$-uniform where $\Delta_j= \Delta^{\eps j}$. 
    Then apply Proposition~\ref{prop: refinement} to $(\cP, \cT)$ with parameter $\Delta$ to obtain a refinement $(\cP_1, \cT_1)$ and the corresponding $(\Delta, s, C_{\Delta}, M_{\Delta})$-nice configuration $(\cP_1^{\Delta}, \cT_1^{\Delta})$ with $C_{\Delta}\approx_{\Delta} \delta^{-\eps^2}$. By Proposition~\ref{prop:s+t=2} with $n=4/\eps$, and $C_0:=\delta^{-\eps^2}< \Delta^{-\eps/8}$, there exists a refinement $(\cQ_{\Delta}, \cT_{\Delta})$ of $(\cP_1^{\Delta}, \cT_1^{\Delta})$  such that 
$|6(\bT^w) \cap \cD_w (\cP_1)| \leapp_{\Delta, \eps} \Delta^{-3\eps/4} \cdot w |\cD_w (\cP_1)|$ for all $w \in \{ 1, \Delta^\eps, \cdots, \Delta \}$ and $\bT \in \cT_\Delta$. Now, we define $\cP_2=\cup_{\bp \in \cQ_\Delta} (\cP_1\cap \bp)$ and for each $p\in \cP_2$, let $\cT_2(p)= \{T\in \cT_1(p): T^{\Delta}\in \cT_{\Delta}(\bp)\}$ and we get from \eqref{it: refinement} of Proposition~\ref{prop: refinement} that  \[
    \sum_{p\in \cP_2} |\cT_2(p)| \approx_{\delta} |\cP|_\delta \cdot M.
    \]

    By pigeonholing, we can find a refinement $(\cP_3, \cT_3)$ of $(\cP_2, \cT_2)$ that is  a configuration refinement of $(\cP, \cT)$ and satisfies $ \cD_{\Delta}(\cT_3(p))\subset  \cT_{\Delta}(p^{\Delta})$ for all $p \in \cP_3$. 

    By Lemma \ref{lem:elementary incidence} on each $Q \in \cD_\Delta (\cP_{3})$ (and assuming $\eps$ is small enough), we can find a refinement $(\cP_4, \cT_4)$ of $(\cP_3, \cT_3)$ such that  for each $T\in \cT_4$ and each $Q\in \cD_{\Delta}(\cP_4)$, 
    \begin{equation}\label{eqn:small tubes}
        |\{ p\in Q\cap \cP_4: T\in \cT_4(p)\}|\lesssim \delta^{-\eps}
    \end{equation}
       By intersecting $\cP_4$ with $\cD_\Delta^{x,y} ([0,1]^2) := \{ [(2k_1 + x) \Delta, (2k_1 + x+1) \Delta ] \times [(2k_2 + y) \Delta, (2k_2 + y+1) \Delta ] :  k_1, k_2 \in \mathbb{N}\cap [0, \frac{1}{2} \Delta^{-1} ]\}$ for suitable $x, y \in \{ 0, 1 \}$, we can also ensure that any  $\Delta$-cubes  $Q_1\neq Q_2\in \cD_\Delta (\cP_4)$, $\text{dist}(Q_1, Q_2)\geq \Delta$ (while preserving the cardinality of $\cP_4$ up to a factor of $4$).
By refining $\cP_4$ further, we may assume $\cP_4$ is $\{ \Delta_j \}_{j=1}^{\eps^{-1}}$-uniform, and \eqref{eqn:small tubes} will still hold.
For each $T\in \cT_4$, there exists  a dyadic number $a(T) \in [1, \delta^{-1}]$ such that  
\begin{equation}\label{eqn: a}
|\{Q\in \cD_{\Delta}(\cP_4): \exists p\in Q\cap \cP_4, T\in \cT_4(p)\}|\sim a(T).
\end{equation}
By pigeonholing, there exists a dyadic number $b(T) \in [1, \delta^{-1}]$ such that there are $a'(T) \approx_{\delta} a(T)$ many $Q\in \cD_{\Delta}(\cP_4)$ with 
\[ |4T\cap \cP_4\cap Q|\sim  b(T). \]
By pigeonholing again, there exists a pair of dyadic numbers $(a,b)$ such that \[ \cT_4':=\{ T\in \cT_4: a(T)\sim a, b(T)\sim b\}\] preserves the number of incidences
\begin{equation}\label{eqn:preserve incidences}
            |\{ (p, T) \in \cP_4 \times \cT_4' : T \in \T_4 (p) \}| \gtrapprox_{\delta} |\cP|_\delta M.
        \end{equation}

   By \eqref{eqn: a} and \eqref{eqn:small tubes}, we get
   \begin{equation}\label{eq: r}
  a\lessapprox_{\delta}  |\{p\in \cP_4: T\in \cT_4(p)\}|\lesim \delta^{-\eps} a. 
    \end{equation}

    Now we verify the assumptions of Lemma~\ref{lem:r-rich}. 
    We have  $| 6(T^w) \cap \cD_w (\cP_4)| \le \Delta^{-\eps} \cdot w^{1-\eps} |\cD_w (\cP_4)|$ since $|\cD_w (\cP_4)| \geapp |\cD_w (\cP)|$ (by uniformity of $\cP$) and the implicit log factor is dominated by $\Delta^{-\eps/2}$ if $\delta_0 (\eps)$ is chosen sufficiently small (recall $\Delta < \delta_0^{8\eps}$). Now, the conditions of Lemma \ref{lem:r-rich} are satisfied, with $\cP_4$ in the place of $\cP$ and  $\inf_{T \in \T_4'} a'(T) \app_\delta a$ in the place of $a$.

    By Lemma~\ref{lem:r-rich}, 
    either $a\lesssim \delta^{1-2\eps}|\cP|_{\delta}$ or 
    \[
    |\cT_4'|\lesssim  \delta^{-5\eps} \frac{|\cP|_{\delta}^2}{a^3 b^2} \lesssim \delta^{-5\eps} \frac{|\cP|_{\delta}^2}{a^3}.
    \]
    On the other hand, by \eqref{eqn:preserve incidences} and \eqref{eq: r}, 
    \begin{equation}\label{eqn:preserve incidences 2}
        |\cT_4'| \gtrapprox_{\delta}  \delta^{\eps} \frac{|\cP|_{\delta} M}{a}.
    \end{equation}
    Thus in the first case $a\lesssim \delta^{1-2\eps}|\cP|_{\delta}$, we get $|\T_4| \ge |\T_4'| \geapp_\delta M \delta^{-1+3\eps}$. In the second case, we get
    \[
    a^2\lessapprox_{\delta} \frac{|\cP|_{\delta}}{M} \delta^{-6\eps}
    \]
    and by plugging this result into \eqref{eqn:preserve incidences 2}, we get $|\T_4| \gtrapprox_{\delta} |\cP|_{\delta}^{1/2} M^{3/2} \cdot \delta^{3\eps}.$ 
\end{proof}

\section{Proof of Lemma \ref{lem:r-rich}}
We restate Lemma \ref{lem:r-rich}, which may have independent interest. 
Recall if $\cP$ is a multi-set of dyadic $\delta$--cubes and $Q\subset \mathbb{R}^2$ is a set, then $|\cP\cap Q|$ means the number of $\delta$--cubes in $\cP$ contained in $Q$ counted with repetition. For any $\delta$-tube $T$ and $b\in \mathbb{N}$,  let $N_{w, b} (T)$ be the number of $Q \in \cD_w (\cP)$ such that $|4T \cap Q \cap \cP| \ge b$. 
\begin{lemma}\label{lem:r-rich'}
 Let $\delta < \frac{\Delta}{32} \in 2^{-\mathbb{N}}$, $\eps \in \frac{1}{\N}$, and let $\cP$ be a set of dyadic $\delta$-cubes (with possible repetitions) such that for all $1 \le k \le \eps^{-1}$,  each $Q \in \cD_{\Delta^{k\eps}} (\cP)$ contains about the same number of cubes in $\cP$, and for any $Q_1\neq  Q_2 \in \cD_\Delta (\cP)$, $\text{dist} (Q_1, Q_2)\geq \Delta$.
    For $a \ge 2$ and $ab > \delta^{1-2\eps} |\cP|$, let  $\T_{a,b}$ be a set of  distinct  dyadic $\delta$-tubes $T$ satisfying  
    \begin{enumerate}
        \item \label{it:good} $|6(T^\rho) \cap \cD_\rho (\cP)| \le \Delta^{-\eps} \cdot \rho |\cD_\rho (\cP)|$ for all $\rho \in \{ \Delta^\eps, \Delta^{2\eps}, \cdots, \Delta \}$,
        \item \label{it:a} $N_{\Delta, b} (T) \ge a$.
    \end{enumerate}
    Then 
    \begin{equation}
        |\T_{a,b}| \lesim_{\eps} \frac{|\cP|^2}{a^3 b^2} \delta^{-5\eps}.
    \end{equation}
\end{lemma}

The proof of Lemma \ref{lem:r-rich'} uses the same high-low method as in \cite{guth2019incidence}, but crucially we streamline the induction on scales mechanism.
First, we begin with a weaker version of Lemma \ref{lem:r-rich'}, which is true under a weaker hypothesis (that $\cP$ is a $(\Delta, 1, K)$-set). It will be useful when $a$ is small. In Lemma~\ref{lem:r-rich weaker}, slightly abusing notation, let $N_{\Delta, b} (T)$ be the number of $Q \in \cD_\Delta (\cP)$ such that $|6T \cap Q \cap \cP| \ge b$. 
\begin{lemma}\label{lem:r-rich weaker}
    Let $\delta < \frac{\Delta}{32}\in 2^{-\mathbb{N}}$. Suppose $\cP \subset [0,1]^2$ is a $(\Delta, 1, K)$-set of not necessarily distinct dyadic $\delta$-cubes: i.e. for all $\Delta \le r \le 1$, we have $|\cP \cap Q| \le K r |\cP|$ for $Q \in \cD_r (\cP)$. Suppose for any $Q_1\neq Q_2\in \cD_\Delta (\cP)$, $\text{dist}(Q_1, Q_2)\geq \Delta$. Then
    \begin{equation*}
        \sum_{T\in \cT^{\delta}: N_{\Delta,b} (T) \ge 2} N_{\Delta,b} (T)^2 \lesim  K |\log \Delta| \cdot \frac{|\cP|^2}{b^2}.
    \end{equation*}
    If $\T_{a,b} \subset \T^\delta$ is a set of essentially distinct dyadic $\delta$-tubes with $N_{\Delta,b} (T) \ge a$ and $a \ge 2$, then \[|\T_{a,b}| \lesim K |\log \Delta| \cdot \frac{|\cP|^2}{(ab)^2}.\]
\end{lemma}

\begin{proof}
    This is a standard $L^2$-energy argument (similar to Lemma \ref{lem:elementary incidence}). Consider $J$, the number of triples $(p_1, p_2, T) \in \cP^2 \times \T^\delta$ such that $p_1, p_2$ belong to different $\Delta$-cubes of $\cD_\Delta (\cP)$ and $p_1,  p_2 \subset  6T$. We have
    \[ J \ge \frac{1}{2} b^2 \sum_{T\in \cT^{\delta}: N_{\Delta,b} (T) \ge 2} |N_{\Delta,b} (T)|^2 \]
    because each $T$ with $N_{\Delta,b} (T) = a \ge 2$ contributes $\ge ab \cdot (a-1)b \ge \frac{1}{2} a^2 b^2$ many triples to $J$. On the other hand, since any $Q_1\neq Q_2\in \cD_\Delta(\cP)$ satisfies $\text{dist}(Q_1, Q_2)\geq \Delta$, we automatically have the distance  $d(p_1, p_2) \in [\Delta, 1]$, so
    \begin{equation*}
        J \lesim \sum_{p_1 \in \cP} \sum_{\rho\in [\Delta, 1] \text{ dyadic}} K|\cP|\cdot  \rho \cdot \frac{1}{\rho} = K \log \Delta^{-1} \cdot |\cP|^2.
    \end{equation*}
    Here, there are $K  \rho |\cP|$ many $\delta$-cubes $p_2\in \cP$ at distance $\sim \rho $ from $p_1$, and there are $\sim \frac{1}{\rho }$ many distinct $\delta$-tubes $T\in \cT^{\delta}$ such  that  $p_1, p_2\subset 6T$ when $\text{dist}(p_1, p_2)\sim \rho$.
\end{proof}

\begin{proof}[Proof of Lemma \ref{lem:r-rich'}]  $\phantom{1}$\\
\noindent {\bf Step 1. A weak estimate. }
    Before anything else, let us check that $\cP$ satisfies the conditions of Lemma \ref{lem:r-rich weaker}. For any $T\in \cT_{a,b}$ and $\rho=\Delta^{k \eps}$, $1\le k\le \eps^{-1}$,  we have $1\leq N_{\rho,b}(T)\leq \Delta^{-\eps} \cdot \rho|\cD_\rho(\cP)|$.  Thus $|\cD_\rho(\cP)|\geq \Delta^{\eps} \cdot \rho^{-1}$. Since each $Q\in \cD_{\Delta^{k\eps}}(\cP)$ contains about the same number of balls in $\cP$, we obtain
    \begin{equation*}
        |\cP \cap Q| \sim  \frac{|\cP|}{|\cD_\rho (\cP)|} \le \Delta^{-\eps} \cdot \rho |\cP|.
    \end{equation*}
    
    The same holds for general dyadic $\Delta \le r \le 1$ at the cost of a $\Delta^{-\eps}$ factor (round $r$ up to the nearest $\Delta^{k\eps}$). Thus, the assumptions of Lemma \ref{lem:r-rich weaker} are satisfied with $K = \Delta^{-2\eps}$, so
    \begin{equation}\label{eqn:weak bound}
        |\T_{a,b}| \lesssim_{\eps} \Delta^{-3\eps} \cdot \frac{|\cP|^2}{(ab)^2}.
    \end{equation}
    This bound is too weak in general, but it does allow us to dispose of a special edge case. Pick $\Delta_0 = \Delta_0 (\eps) > 0$ a small constant (to be chosen later). 
    If $\Delta > \Delta_0$, then $a \le \Delta^{-1} \le \Delta_0^{-1}$ to ensure $\T_{a,b} \neq \emptyset$. Thus, \eqref{eqn:weak bound} will suffice (upon absorbing $\Delta_0^{-1}$ into the implicit constant in $\lesssim_{\eps}$).

\medskip 

\noindent {\bf Step 2. Base cases. }
    For a fixed $\Delta < \Delta_0$, we induct on $\delta = \Delta 2^{-n}$.  There are two base cases.

The first base case is $\delta > \Delta^{1+\eps}$, and we will prove $|\T_{a,b}| = 0$ using the $\delta^{-2\eps}$ factor in $ab>\delta^{1-2\eps}|\cP|$.  Indeed, if $T \in \T_{a,b}$, then $|\cP \cap Q| \ge b$ for some $Q \in \cD_\Delta (\cP)$, so by uniformity of $\cP$ at scale $\Delta$, we get $|\cP| \gtrsim b |\cD_\Delta (\cP)|$.  Also,  by  item \eqref{it:a}, 
\[ N_{\Delta, b} (T) \ge a > \delta^{1-2\eps} b^{-1} |\cP| \gtrsim \delta^{1-2\eps} |\cD_{\Delta}(\cP)|.\]
On the other hand, by item \eqref{it:good}, \[N_{\Delta, b} (T) \leq |6(T^{\Delta})\cap \cD_{\Delta}(\cP)| \leq \Delta^{-\eps} \cdot \Delta |\cD_{\Delta}(\cP)|\] because $N_{\Delta, b}(T)$ is bounded by the number of $Q\in \cD_{\Delta}(\cP)$ contained in  $6(T^{\Delta})$.    When  $\Delta < \Delta_0 (\eps)$ is small enough, $\Delta^{1-\eps -2\eps^2} < \delta^{1-2\eps}$ because $\delta >\Delta^{1+\eps}$. This yields a contradiction as $  \Delta^{1-\eps} |\cD_{\Delta}(\cP)| \geq  N_{\Delta, b}(T)\gtrsim  \delta^{1-2\eps} |\cD_{\Delta}(\cP)|$.  Thus, $|\T_{a,b}| = 0$, which completes the  first base case. 
    
The second base case is when $a$ is small: if $a \le 2\Delta^{-2\eps}$, then \eqref{eqn:weak bound} will give us the desired bound.

\medskip 

\noindent {\bf Step 3. Inductive step and choosing an intermediate scale}
   From now on, we can assume $\delta < \Delta^{1+\eps}$ and $a > 2\Delta^{-2\eps}$. For the inductive step, assume true for $\delta > \Delta 2^{-n+1}$, and consider $\delta =\Delta 2^{-n}$. We shall assume $\T_{a,b} \neq \emptyset$, as otherwise the lemma is obviously true. 

    We will first choose an intermediate scale $w$ for later use in Step 4.
    Let $k$ be the \textbf{largest} integer $1\leq k \leq \eps^{-1}$  such that 
    \begin{equation} \label{eq: a_lowerbd}
    a > 2\Delta^{-\eps} \cdot w|\cD_w (\cP)| \end{equation}
    for $w=\Delta^{k \eps}$. Such $k$ must exist because when $k=1$, this corresponds to our assumption $a> 2\Delta^{-2\eps}$.  Thus $k\geq 1$. We also have $k<\eps^{-1}$, since otherwise $ a > 2\Delta^{1-\eps} |\cD_{\Delta}(\cP)|$, contradicting item~\eqref{it:good}.
    
This choice of scale $w = \Delta^{k\epsilon}$ may seem opaque but is crucial for the remainder of the proof. We obtain two useful conclusions from $a \leq 2\Delta^{-\eps} \cdot \rho|\cD_\rho (\cP)|$ for any $\rho = \Delta^{m \eps}$, $m > k$: 
    \begin{enumerate}[(C1)]
        \item\label{conc1} $a < 2\Delta^{-\eps} \cdot (w\Delta^\eps) |\cD_{w\Delta^\eps} (\cP)| \lesim \Delta^{-2\eps} \cdot w|\cD_w (\cP)|$, since each $Q \in \cD_w (\cP)$ contains $\lesim \Delta^{-2\eps}$ many elements of $\cD_{w\Delta^\eps} (\cP)$. This combines with the lower bound \eqref{eq: a_lowerbd} yields 
        \begin{equation}\label{eq: a}
           2\Delta^{-\eps} \cdot w|\cD_w(\cP)|<a < \Delta^{-2\eps} \cdot w|\cD_w(\cP)|.
        \end{equation}

        \item\label{conc2} For any $Q \in \cD_w (\cP)$, we have that $\cP^Q = \phi_Q (\cP \cap Q)$ is a $(\frac{\Delta}{w}, 1, \Delta^{-\eps})$-set, where $\phi_Q$ is the affine transform that maps $Q$ to $[0,1)^2$. This is because for $\rho = \Delta^{m\eps}$, $m > k$, we have $w|\cD_w(\cP)| \le \frac{a}{2\Delta^{-\eps}} \le \rho |\cD_\rho (\cP)|$. Thus, for any $\tilde{Q} \in \cD_\rho (\cP)$ and $Q\in \cD_{w}(\cP)$,
        \begin{equation*}
            |\tilde{Q} \cap \cP| \sim \frac{|\cP|}{|\cD_\rho (\cP)|} \le \frac{|\cP|}{|\cD_w (\cP)|} \cdot \frac{\rho}{w} \sim  |Q \cap \cP| \cdot \frac{\rho}{w}.
        \end{equation*}
        By rounding $r$ up to the nearest $\Delta^{m\eps}$, we have established that for all $\tilde{Q} \in \cD_r (\cP)$, $\Delta \le r \le w$, that $|\tilde{Q} \cap \cP| \le \Delta^{-\eps} |Q \cap \cP| \cdot \frac{\rho}{w}$. 
    \end{enumerate}
In the remainder of the proof, we only use properties \ref{conc1} and \ref{conc2} of the scale $w = \Delta^{k\epsilon}$.

\medskip 

\noindent {\bf Step 4.  Incidences between tubelets and tubes using high-low method (Proposition~\ref{prop:high low}).  }
    Recall the intermediate scale $w$ chosen in Step 3. 
    For each $Q\in \cD_{w}(\cP)$, let $\U_Q \subset \{ T\cap Q: T\in \cT_{a,b}\}$ be   a maximal essentially distinct set of $\delta \times w$-tubelets: for any $u_1, u_2\in \U_Q$, $|u_1\cap u_2|\leq \frac{1}{4}|u_1|$.  Let $\U=\cup_{Q\in \cD_{w}(\cP)} \U_Q$. 
    For $u \in \U$, define $N_{\Delta, b} (u)$ to be the number of $Q \in \cD_\Delta (\cP)$ such that $|6u \cap Q \cap \cP| \ge b$ (here $6u$ means the $4\delta$-neighborhood of $u$)  and $m(u)$ be the number of $\delta$-tubes in $T\in \T_{a,b}$ such that $|T\cap u |> |u|/4$. We shall assign $u$ the weight $N_{\Delta, b} (u)$, and consider weighted incidences. With this convention, we have
    \begin{equation*}
        a |\T_{a,b}| \le I(\U, \T_{a,b}) := \sum_{u\in \U} N_{\Delta, b} (u)\cdot  |\{ T\in \cT_{a,b}: |u\cap T|>|u|/4\}| = \sum_{u \in \U} N_{\Delta, b}(u) \cdot m(u)
    \end{equation*}
    and by \eqref{eq: a_lowerbd} and Item~\eqref{it: good},
    \begin{equation*}
        \frac{a}{2} |\T_{a,b}| \ge \sup_{T\in \T_{a,b}} |6(T^w) \cap \cD_w (\cP)| \cdot |\T_{a,b}| \ge \sum_{u \in \U} m(u),
    \end{equation*}
    so at least half of the incidences between tubes and tubelets involves tubelets with $N_{\Delta,b} (u) \ge 2$. We follow the calculation in \cite[Theorem 5.1.2]{bradshaw2022additive}. Replace $\U$ by the tubelets $u \in \U$ with $N_{\Delta, b} (u) \ge 2$.

    Let $\T_{\Box}$ be an essentially distinct cover of $[0, 1]^2$ by $\frac{6\delta}{w}$-tubes, such that each $\delta$-tube lies in at least $1$ and at most $O(1)$ many members of $\T_{\Box}$. For each $\bT \in \T_{\Box}$,  let $\U_{\bT}$ denote the set of tubelets $u\in \U$ such that $u\subset \bT$ and the coreline of $u$ is parallel to the coreline of $\bT$ up to an error of size $6\delta/w$. In particular, the tubelets $u\in \U_\bT$ are parallel up to an error of size $12\delta/w$ and  essentially distinct, which means they are $\leq 48$-overlapping (at most $48$ tubelets in $\U_\bT$ intersect at a common point\footnote{because if 5 tubes are parallel up to $\frac{\delta}{w}$ and intersect a common point, then two are not essentially distinct}). After rescaling $\bT$ to $[0,1]^2$, the tubelets in $\U_{\bT}$ become $w$-cubes and the $\delta$-tubes in $\T_{\bT} := \cT_{a,b}\cap \bT$  become $w$-tubes.

 We may decompose our weighted incidence count (with weight $N_{\Delta, b}(u)$ on each $u$) as follows:
    \begin{equation*}
        I(\U, \T_{a,b}) \le \sum_{\bT \in \T_{\Box}} I(\U_{\bT}, \T_{\bT}).
    \end{equation*}
    By Proposition \ref{prop:high low}  with $S = \Delta^{-\eps/100} \in (w^{-\eps/100}, w^{-1})$, we obtain
    \begin{align*}
        a |\T_{a,b}|& \lesim I(\U, \T_{a,b}) = \sum_{\bT \in \T_{\Box}} I(\U_{\bT}, \T_{\bT}) \\
                &\lesim (S w^{-1})^{1/2} \underbrace{\sum_{\bT \in \T_{\Box}} |\T_{\bT}|^{1/2} \left( \sum_{u \in \U_{\bT}} N_{\Delta, b}(u)^2 \right)^{1/2}}_{(A)} + S^{-1+\eps} \underbrace{\sum_{\bT \in \T_{\Box}} I(\U_\bT, \T_\bT^{S\delta})}_{(B)},
    \end{align*}
    where $\cT_{\bT}^{S\delta}=\{ T^{S\delta}: T\in \cT_{\bT}\}$ is counted with multiplicity (see Proposition~\ref{prop:high low} and the explicit formula of $I(\U_{\bT}, \cT_{\bT}^{S\delta})$ below). 

    \medskip 

    \noindent {\bf Step 5. High-frequency case. }

Using Cauchy-Schwarz and the fact that $N_{\Delta, b}(u)\geq 2$ for any $u\in \U$, the high-frequency term $(A)$ is bounded by 
\begin{align*}
(A) & \leq \left( \sum_{\bT \in \T_{\Box}} |\T_{\bT}| \right)^{1/2} \left( \sum_{\bT \in \T_{\Box}} \sum_{\substack{u \in \U_{\bT}\\N_{\Delta,b} (u) \ge 2}} N_{\Delta, b} (u)^2 \right)^{1/2}\\
& \leq |\cT_{a,b}|^{1/2} \left( \sum_{Q \in \cD_w (\cP)} \sum_{\substack{u \in \U_Q\\N_{\Delta, b} (u) \ge 2}} N_{\Delta,b} (u)^2 \right)^{1/2}.
\end{align*}
By Conclusion \ref{conc2} and Lemma \ref{lem:r-rich weaker} applied to $\cP^Q$ for each $Q \in \cD_w (Q)$ and $|\cP\cap Q|\sim \frac{|\cP|}{|\cD_w(P)|}$, we have
    $$ (A) \lesim  |\T_{a,b}|^{1/2} \left( |\cD_w (\cP)| \cdot \Delta^{-2\eps} \left( \frac{|\cP|}{b|\cD_w (\cP)|} \right)^2 \right)^{1/2}.$$
    
If the high-frequency term  dominates, i.e.  $ a|\cT_{a,b}|\lesssim_{\eps} (S w^{-1})^{1/2}(A)$, we get by $S=\Delta^{-\eps/100}$ and \eqref{eq: a}
    \begin{equation*}
        |\T_{a,b}| \lesim \Delta^{-2\eps} S \cdot \frac{w^{-1}}{a^2 b^2} |\cD_w (\cP)| \cdot \left( \frac{|\cP|}{|\cD_w (\cP)|} \right)^2 \lesim \Delta^{-5\eps} \frac{|\cP|^2}{a^3 b^2}.
    \end{equation*}

    \medskip 
    \noindent {\bf Step 6. Low-frequency case. }

Now suppose the low-frequency term  dominates, i.e. \[ a|\cT_{a,b}|\lesssim S^{-1+\eps} \sum_{\bT \in \T_{\Box}} I(\U_\bT, \T_\bT^{S\delta}).\] 
Unwinding the definition of $I(\U_\bT, \T_\bT^{S\delta})$, 
\begin{align*}
(B) & =\sum_{\bT \in \T_{\Box}} I(\U_\bT, \T_\bT^{S\delta}) \\
& \leq \sum_{\bT \in \T_{\Box}} \sum_{u\in \U_\bT} \sum_{T\in \cT_{\bT}: u\subset 6( T^{S\delta})} N_{\Delta, b}(u)\\
& \leq \sum_{\bT\in \cT_{\Box}} \sum_{T\in \cT_{\bT}} \sum_{u\in \U_{\bT}\cap 6(T^{S\delta})} N_{\Delta, b}(u).
\end{align*}

For any  $T\in  \cT_\bT$, let $f(T) := \sum_{u \in \U_{\bT} \cap 6(T^{S\delta})} N_{\Delta, b} (u)$; then $\sum_{\bT \in \T_{\Box}} I(\U_\bT, \T_\bT^{S\delta}) = \sum_{T \in \T_{a,b}} f(T)$. Hence, there exists a subset $\T''_{a,b} \subset \T_{a,b}$ with $|\T''_{a,b}| \ge \frac{1}{2} |\T_{a,b}|$ such that for each $T \in \T''_{a,b}$, we have 
    a lower bound for sum of weights,
    \begin{equation*}
        f(T) = \sum_{u \in \U_{\bT} \cap 6(T^{S\delta})} N_{\Delta, b} (u) \gesim S^{1-\eps} a.
    \end{equation*}

    Now, for $Q \in \cD_\Delta (\cP)$, we let $g_T (Q)$ be the number of $u \in \U_{\bT} \cap 6(T^{S\delta})$ such that $|6u \cap Q \cap \cP| \ge b$.  Recall from Step 4 that the tubelets $u\in \mathbb{U}_{\bT}$ are $\leq 48$-overlapping. We also have the simple geometric observation $6T^{S\delta}\subset 3T^{2S\delta}$. So we have  \begin{equation} 
    \label{eq: lowerbdkb} 
    |4(T^{2S\delta})\cap Q\cap \mathcal{P}^{2S\delta}| \geq  \frac{g_T(Q) b}{48},\end{equation} where $\cP^{2S\delta}$ is counted with multiplicity. By double counting and applying the definition of $N_{\Delta, b} (u)$ from Step 4, we have
    \[ \sum_{Q \in \cD_{\Delta}(\cP)} g_T (Q) = f(T) \gtrsim S^{1-\eps}a.\]
    We also have $g_T (Q) \lesim S$ since at most that many tubelets in $\U_{\bT}$ can intersect $6(T^{S\delta}) \cap Q$, which is roughly a $6S\delta \times \Delta$ rectangle. Indeed, from Step 4, all tubelets $u\in \mathbb{U}_{\bT}$ are parallel to the coreline of $\bT$ up to an error of $6\delta/w$. Since $w\geq \Delta$, at most $S$ tubelets in $\mathbb{U}_\bT$ can intersect $6(T^{S\delta})\cap Q$. 


    Hence, by pigeonholing, we can find a dyadic $k_T \in [1, S]$ and a set $\cQ_T \subset \cD_\Delta (\cP)$ with $|\cQ_T| \gtrsim (\log S)^{-1} \cdot \frac{S^{1-\eps} a}{k_T}$ such that $g_T (Q) \sim  k_T$ for all $Q \in \cQ_T$. By another pigeonholing, we can further choose a dyadic number $k$ and  a subset $\T'_{a,b} \subset \T''_{a,b}$ with $|\T'_{a,b}| \gesim (\log S)^{-1} |\T''_{a,b}|$ such that $k_T = 64k$ is the same for all $T \in \T'_{a,b}$. Finally, we consider the set $\tilde\cT := \{ T^{2S\delta} : T \in \T'_{a,b} \}$ (where we delete all duplicate elements); it has cardinality $\gtrsim S^{-2} |\T'_{a,b}|$ because at most $\lesssim S^{-2}$ many dyadic $\delta$-tubes $T$ share a common dyadic ancestor $T^{2S\delta}$.
    
    %
    We aim to apply the inductive hypothesis to  $\tilde\cT$. 
    We first check that Item \ref{it:good} is satisfied for $\tilde\cT$. Indeed, pick $T \in \T'_{a,b}$; then
    \begin{equation*}
        |6((T^{2S\delta})^\rho) \cap \cD_\rho (\cP)| = |6(T^\rho) \cap \cD_\rho (\cP)| \le \Delta^{-\eps} \cdot \rho |\cD_\rho (\cP)| \text{ for all } \rho \in \{ \Delta^\eps, \Delta^{2\eps}, \cdots, \Delta \}
    \end{equation*}
    Here, we have $(T^{2S\delta})^\rho = T^\rho$ since $\rho \geq \Delta \ge 2S\delta$, using the fact that $\delta\leq \Delta^{1+\eps}$ and  $S = \Delta^{-\eps/100}$. Furthermore, for every $T^{2S\delta} \in \tilde\cT$, by \eqref{eq: lowerbdkb}, we have $N_{\Delta,kb} (T^{2S\delta}) \ge a' := \frac{S^{1-\eps} a}{k \log S}$.  Here in our notation, 
    $\tilde{\cP}:=\{p^{2S\delta}: p\in \cP\}$ is a set of dyadic $2S\delta$-cubes (with multiplicity, so $|\cP|=|\tilde{\cP}|$) and $N_{\Delta, kb}(T^{2S\delta})$ means the number of $Q\in \cD_{\Delta}(\tilde{\cP})=\cD_\Delta(\cP)$ such that $|4(T^{2S\delta})\cap Q\cap \cP^{2S\delta}|\geq kb$.  Let $b' := kb$. Since $k \lesim S$, we get $b'\lesim Sb$. Recall $a\geq 2\Delta^{-2\eps}$, we have $a' \gtrsim S^{-\eps} (\log S)^{-1} a \ge 2$. Also $a'b' \gtrsim S^{1-\eps} (\log S)^{-1} ab > \max\{ (S\delta)^{1-2\eps} |\cP|, 2\}$, we can apply the inductive hypothesis to $\tilde{\T}$ and $\tilde{\cP}$ to get
    \begin{equation*}
        |\T_{a,b}| \lesim S^2 \log S \cdot |\tilde{\T}| \lesim_{\eps} S^{2} \log S \cdot \frac{|\cP|^2}{(a')^3 (b')^2} (S\delta)^{-5\eps} \lesim_{\eps} (\log S)^4 S^{-2\eps} \frac{|\cP|^2 b'}{S a^3 b^3} \delta^{-5\eps}  \le S^{-\eps} \cdot \frac{|\cP|^2}{a^3 b^2} \delta^{-5\eps}.
    \end{equation*}
    The $S^{-\eps} = \Delta^{-\eps^2/100}$ cleans up the losses from $\lesim_\eps$ if $\Delta < \Delta_0 (\eps)$. This closes the inductive step and finishes the proof of Lemma \ref{lem:r-rich}.

\end{proof}

\section{General case using Orponen-Shmerkin's regular Furstenberg result}



In the previous sections, we dealt with the Furstenberg set problem for a semi-well spaced set $\cP$. Now we turn to the other extreme, when $\cP$ is almost AD-regular. This case has been solved by Orponen-Shmerkin \cite[Theorem 5.7]{orponen2023projections}, and we present a slight refinement.
\begin{theorem}\label{prop:furst-ADreg}
    For any $\eps > 0, s\in (0,1], t\in [s,2-s]$, there exists $\eta(\eps,s) > 0$ and $\delta_0=\delta(\eps,s,t)>0$ such that for any $\delta\in (0, \delta_0)$ and any $(\delta, s, \delta^{-\eta}, M)$-nice configuration $(\cP, \cT)$ with $\cP$ $(t, \eta)$-almost AD regular as in Definition~\ref{def: AD}, then
    \[
\frac{|\cT|_{\delta}}{M} \gtrsim \delta^{-\frac{s+t}{2}+\eps}.
\]  
\end{theorem}

\begin{remark}
    There are two differences between our proposition and \cite[Theorem 5.7]{orponen2023projections}. First, our definition of almost AD-regular set is slightly different from the $(\delta, s, C)$-regular set in \cite[Definition 4.2]{orponen2023projections}. Second, our $\eta$ is only required to depend on $\eps, s$, while the $\eta$ in \cite[Theorem 5.7]{orponen2023projections} may depend on $\eps, s, t$.
\end{remark}

\begin{proof}
    We resolve the differences mentioned in the remark. First, we show that if $\cP$ is $(t, \eta)$-almost AD-regular set, then it is also a $(\delta, t, \delta^{-3\eta})$-regular set in the sense of \cite[Definition 4.2]{orponen2023projections}. Indeed, the upper bound in Definition \ref{def: AD} shows that $\cP$ is a $(\delta, t, 4\delta^{-\eta})$-set. (More precisely, if $\cP \cap B(x, r) \neq \emptyset$ for some $x \in \R^2$, then let $y \in \cP \cap B(x, r)$; then $\cP \cap B(x, r) \subset \cP \cap B(y, 2r)$.) Furthermore, for any dyadic $\delta \le r \le R \le 1$ and $\bp \in \cD_R(\cP)$, we have
    \begin{equation*}
        4\delta^{-\eta} \cdot R^t \cdot |\cP|_\delta \ge |\cP \cap \bp|_\delta = \sum_{Q \in \cD_r (\cP \cap \bp)} |Q \cap \cP|_\delta \gesim |\cP \cap \bp|_r \cdot \delta^\eta r^t |\cP|_\delta.
    \end{equation*}
    The last inequality uses the lower bound in Definition \ref{def: AD} and the fact that for any $Q \in \cD_r (\cP \cap \bp)$ and $x \in \bp \cap Q$, we have $Q \subset B(x, \sqrt{2} r)$. Hence, we get $|\cP \cap \bp|_r \lesim \delta^{-2\eta} \cdot \left( \frac{R}{r} \right)^t$, and the implicit constant can be absorbed into $\delta^{-\eta}$ if $\delta > 0$ is sufficiently small.

    The next difference is that the $\eta$ in \cite[Theorem 5.7]{orponen2023projections} depends a priori on $\eps, s, t$. To prevent confusion later, we let $\tilde{\eta} (\eps, s, t)$ denote this parameter, and reserve $\eta(\eps, s)$ for the desired parameter which is independent of $t$.
    
    To remove the dependence on $t$, we use a compactness argument. For fixed $s$, the set $S = [s, 2-s] \subset \R$ is compact (note that the end points $t=s$ and $t=2-s$ are almost AD-regular cases of Lemma~\ref{lem:elementary incidence} and Proposition~\ref{prop:s+t=2}). Now let $\eta_t = \frac{1}{2} \min \{\tilde{\eta}(\eps/2, s, t), \eps \} > 0$ and define $B_t := B(t, \eta_t)$ for $t \in S$. Then $S = \cup_{t \in S} B_t$, so $S$ admits a finite subcover $\cup_{t \in S'} B_t$ where $S' \subset S$ is finite. Now define $\eta = \eta(\eps, s) := \min_{t \in S'} \eta_t$. We claim that this choice of $\eta$ works for any $t \in S$. Indeed, given $t \in S$, let $t' \in S'$ such that $|t - t'| \le \eta_{t'}$. Then since $\eta + |t - t'| \le 2\eta_{t'} \le \tilde{\eta}(\eps/2, s, t)$, we see that a $(\delta, t, \delta^{-\eta})$-set $\cP$ is also a $(\delta, t', \delta^{-\tilde{\eta}(\eps/2, s, t')})$-set. Hence, by \cite[Theorem 5.7]{orponen2023projections} and $t' \ge t - \eta_{t'} \ge t - \eps/2$, we know that $\frac{|\cT|_{\delta}}{M} \gtrsim \delta^{-\frac{s+t'}{2}+\eps/2} \ge \delta^{-\frac{s+t}{2}+\eps}.$
\end{proof}

Thus, we know that Theorem \ref{Thm: main} is true when the set $\cP$ is either semi-well spaced or AD-regular. To prove Theorem \ref{Thm: main} for general sets $\cP$, we show that the set of scales $(\delta, 1)$ can be partitioned into two sets $A, B$ such that $\cP$ looks AD-regular on scales in $A$ and looks semi-well spaced on scales in $B$. To formalize this intuition, we first define the branching function of a set $\cP$, c.f. \cite[Definition 2.22]{orponen2023projections}.

\begin{definition}[Branching function]
    Let $T\in \mathbb{N}$, and let $\cP\subset [0,1)^d$ be a $\{\Delta_j\}_{j=1}^n$-uniform set, with $\Delta_j: = 2^{-jT}$, and with associated sequence $\{N_j\}_{j=1}^n\subset \{ 1, \dots, 2^{dT}\}^n$. 
    We define the \emph{branching function} $f: [0, n]\rightarrow [0, dn]$ by setting $f(0)=0$, and 
    \[
    f(j):=\frac{\log |\cP|_{2^{-jT}}}{T} =\frac{1}{T}\sum_{i=1}^j \log N_i, \quad i \in \{1, \dots n\}. 
    \]
\end{definition}


We now try to convert the notions of $(\delta, s, C)$-set, AD-regular set, and semi-well spaced set into the language of branching functions. Let $s_f (a, b) = \frac{f(b) - f(a)}{b - a}$ denote the slope of a line segment between $(a, b)$ and $(f(a), f(b))$. We say that a function $f: [0,n]\rightarrow \mathbb{R}$ is $(\sigma, \eps)$-superlinear on $[a,b]\subset [0,n]$, or that $(f, a, b)$ is $(\sigma, \eps)$-superlinear, if 
 \[ f(x)\geq f(a) + \sigma (x-a) - \eps (b - a), \qquad x\in [a, b]. 
 \]
If $\sigma = s_f (a, b)$, then we just say that $(f, a, b)$ is $\eps$-superlinear.
Since $f(a) + s_f (a,b) (x-a) = f(b) - s_f (a,b)(b-x)$, another criterion for $(f, a, b)$ being $\eps$-superlinear is
\begin{equation*}
    f(x) \geq f(b) - s_f (a,b)(b-x) - \eps (b - a), \qquad x \in [a, b].
\end{equation*}
We will use both criteria in later proofs. Finally, $(f, a, b)$ is $\eps$-linear if
\[ |f(x) - f(a) - s_f(a,b) (x-a)| \le \eps (b - a), \qquad x\in [a, b]. 
 \]

If $\cP$ is a $(\delta, s, \delta^{-\eps})$-set that is $\{ \Delta_j \}_{j=1}^n$-uniform, then one can show that the branching function of $f$ is $(s, \eps)$-superlinear. Likewise, if $\cP$ is $(s, \eps)$-almost AD-regular, then $f$ is $\eps$-linear. The reader is encouraged to ponder what happens if $\cP$ is semi-well spaced in the sense of Proposition \ref{prop: key} (the answer is revealed in \eqref{eq: fsemiwellspaced} with $u=2-s$ and Remark \ref{rem:dictionary}).

The following proposition is a ``merging'' of \cite[Lemma 2.6 and 2.7]{orponen2023projections} and may be of independent interest. Roughly, it allows us to partition the set of scales $[\delta, 1]$ into a bounded number of intervals, where in each interval the set of scales look either almost AD-regular or semi-well spaced. 

\begin{prop}\label{prop:multiscale}
    Fix $0<s<t<u$ and $m > 0$. For every $0 < \eps < \min(u-s, \frac{1}{2})$, there is $\tau = \tau(\eps, s,u, t) > 0$ such that the following holds: for every piecewise affine $2$-Lipschitz function $f : [0, m] \to \R$ such that
    \begin{equation*}
        f(x) \ge tx - \eps^2 m \text{ for all } x \in [0, m], \quad  \text{ and } f(m)\leq (t+\eps^2)m, 
    \end{equation*}
    there exists a family of non-overlapping intervals $\{ [c_j, d_j] \}_{j=1}^n$ contained in $[0, m]$ such that:
    \begin{enumerate}
        \item For each $j$, at least one of the following alternatives holds:
        \begin{enumerate}
            \item $(f, c_j, d_j)$ is $\eps$-linear with $s_f(c_j,d_j)\in [s,u]$;

            \item $(f, c_j, d_j)$ is $\eps$-superlinear with $s_f (c_j, d_j)\in [s, u]$ and 
            \begin{equation}\label{eq: fsemiwellspaced} 
            f(x) \ge \min \{ f(c_j) + u(x - c_j), f(d_j) - s(d_j - x)\} - \eps(d_j - c_j).
            \end{equation}
        \end{enumerate}

        \item $d_j - c_j \ge \tau m$ for all $j$;

        \item $|[0, m] \setminus \cup_j [c_j, d_j]| \lesim_{s,t,u} \eps m$.
    \end{enumerate}
\end{prop}

\begin{remark}\label{rem:dictionary}
    Let $f$ be the branching function of a $\{\Delta_j\}_{j=1}^n$-uniform set $\cP$ with $\Delta_j:=2^{-jT}$.  Write $\Delta=\Delta_1$, then $\delta=\Delta^n$. Let $[c_j, d_j]$ be an interval satisfying condition (1)(a) in Proposition \ref{prop:multiscale}. Let $\bp \in \cD_{\Delta^{c_j}}(\cP)$ and $\cP_j = \phi_\bp (\cP \cap \bp)$ (recall $\phi_{\bp}$ is the affine transform that maps $\bp$ to $[0,1)^2$), which encodes the geometry of $\cP$ between scales $\Delta^n$ and $\Delta^{c_j}$. Then for $x \in (c_j, d_j)$, write $s_f (x, d_j) = t_j$, so $\cP_j$ is a $(\Delta^{d_j - c_j}, t_j, C)$-AD regular set with $C\approx \Delta^{ -\eps(d_j-c_j)}$ up to a constant power of $\Delta$.  Since we only care about $\cP_j$ upto resolution $\Delta^{d_j-c_j}$, we can also identify $\cD_{\Delta^{d_j-c_j}} \cP_j$ with $\cP_j$.  

    If $[c_j, d_j]$ satisfies condition (1)(b), then let $\bp\in \cD_{\Delta^{c_j}}(\cP)$  and $\cP_j = \phi_{\bp}(\cP\cap \bp)$ as before.  Pick $\tilde{\Delta}=\Delta^{c-c_j}$ where $c$ satisfies $f(c_j)+u(c-c_j) = f(d_j)-s(d_j-c)$. Then for $x \in (c, d_j)$, we  have $f(x)\geq f(d_j) -s(d_j-x) -\eps(d_j-c_j)$, which means that $\cP_j$ is a $(\Delta^{d_j-c_j}, s, C)$-Katz-Tao set ``from scales $\Delta^{d_j-c_j}$ to $\tilde{\Delta}$'' with $C\approx \Delta^{-\eps(d_j-c_j)}$ up to a constant power of $\Delta$.  Likewsie, for $x\in (c_j, c)$, we have $f(x)\geq f(c_j) +u(x-c_j) -\eps(d_j-c_j)$, which means that $\cP_j$ is a $(\tilde{\Delta}, u, C)$-set. 
    Thus, $\cP_j$ is semi well-spaced in the sense of Proposition \ref{prop: key}. See Lemma \ref{lem:dictionary} for more precise statements.
\end{remark}

\begin{proof}
    By \cite[Lemma 2.7]{orponen2023projections} for $\eps^2$ and using $t < u$, there is a family of intervals $\{ [c_j, d_j] \}_{j=1}^n$, where $d_{j-1}\leq c_j$,   contained in $[0, m]$ such that:

    \begin{enumerate}
        \item For each $j$, at least one of the following alternatives holds: 
        \begin{enumerate}[(i)]
            \item $(f, c_j, d_j)$ is $\eps^2$-linear with $s_f (c_j, d_j) \in [0, u]$;

            \item $(f, c_j, d_j)$ is $\eps^2$-superlinear with $s_f (c_j, d_j) = u$.
        \end{enumerate}

        \item $d_j - c_j \ge \tau m$ for all $j$;

        \item $|[0, m] \setminus \cup_j [c_j, d_j]| \lesim_{t,u} \eps^2 m$.
    \end{enumerate}
    This is close to our desired output, except the first alternative allows $s_f (c_j, d_j) \in [0, u]$ instead of $[s, u]$. To remedy this, we follow the algorithm in \cite[Lemma 8.5]{orponen2021hausdorff}. Initially set $k = n$; the value of $k$ represents that we still need to process intervals $[c_j, d_j]$ with $j \le k$. For purposes that will be apparent later, we shall weaken the condition on the last unprocessed interval $[c_k, d_k]$: for $(f, c_n, d_n)$, we weaken (i) to (i') $(f, c_k, d_k)$ is $\eps$-linear with $s_f(c_k, d_k) \in [0, u]$. Clearly (i) implies (i') (as long as $\eps < 1$).

    If $[c_k, d_k]$ satisfies (ii) or $t_k \ge s$ in (i'), then $[c_k, d_k]$ is an interval of type (b) or type (a) respectively, as in (1) of the statement of Proposition \ref{prop:multiscale}, and so we shall mark $[c_k, d_k]$ as ``processed'' and set $k$ to $k-1$. So, assume $[c_k, d_k]$ is of type (i') and satisfies $t_k < s$. If $d_k \le \frac{\eps m}{t-s}$, then remove all intervals $[c_i, d_i]$ with $i < k$ and we are done. Otherwise, we observe that
    \begin{equation*}
        s_f (0, d_k) = \frac{f(d_k)}{d_k} \ge \frac{td_k - \eps m}{d_k} > s.
    \end{equation*}
    Since $s_f (c_k, d_k) < s$ and $f$ is piecewise affine, we can find a largest $c' \in (0, c_k)$ such that $s_f (c', d_k) = s$.

    Now we check that $(f, c', d_k)$ is $\eps$-superlinear. If $x \in (c', c_k)$, then by maximality of $c'$, we have $s_f (x, d_k) < s = s_f (c', d_k)$. Since $s_f (c', d_k)$ is a convex combination of $s_f (c', x)$ and $s_f (x, d_k)$, we get $s_f (c', x) \ge s$, so $f(x) \ge f(c') + s(x - c')$. If $x \in (c_k, d_k)$, then we have, using $t_k:=s_f (c_k, d_k) < s$ and $f$ being $\eps$-linear on $[c_k, d_k]$,
    \begin{align}
        f(x) &\ge f(d_k) - t_k (d_k - x) - \eps (d_k - c_k) \nonumber\\
            &\ge f(d_k) - s (d_k - x) - \eps (d_k - c') \label{align1}
    \end{align}
    Hence, since $s_f (c', d_k) = s$, we get that $(f, c', d_k)$ is $\eps$-superlinear. 
    

    Observe that $[c', d_k]$ is a type $(b)$ interval (degenerate case) and can serve as part of a type $(b)$ interval. 
    
    We consider several cases:
    \begin{itemize}
        \item If $c' \notin \cup_j [c_j, d_j]$, then repeat the procedure with $k$ equal to the largest index with $d_k < c'$ (if $c' < c_1$ or no such $k$ exists, we stop).

        \item If $c' \in [c_\ell, d_\ell]$ and $(c' - c_\ell) \le 2\eps (d_\ell - c_\ell)$, we discard the piece $[c_\ell, c']$. Repeat the procedure, setting $k := \ell - 1$.

        \item If $c' \in [c_\ell, d_\ell]$, $(c' - c_\ell) \ge 2\eps (d_\ell - c_\ell)$, and $[c_\ell, d_\ell]$ is type (i), then we replace the interval $[c_\ell, d_\ell]$ by $[c_\ell, c']$ (which satisfies (i') as we show next)  and repeat the procedure, setting $k := \ell$. Indeed,  $(f, c_\ell, c')$ is $ \eps$-linear since $(f, c_\ell, d_\ell)$ is $\eps^2$-linear and $(c'-c_\ell)\geq 2\eps(d_\ell-c_\ell)$. To see this, write $t_\ell= s_f(c_\ell, d_\ell)$, then $|f(c')-f(c_\ell) - t_\ell (c'-c_\ell)| \leq \eps^2(d_\ell-c_\ell) \leq \frac{\eps}{2}(c'-c_\ell)$, so for any $x\in [c_\ell, c']$,  $ |s_f(c_\ell, c') - t_\ell|\leq \eps/2$ and 
        \[|f(x)-f(c_\ell) - s_f(c_\ell, c') (x-c_\ell)| \leq \frac{\eps}{2}(x-c_\ell)+\frac{\eps}{2}(c'-c_\ell)\leq \eps (c'-c_\ell).\]  
        Hence, (i') is satisfied in the procedure.

        \item If $c' \in [c_\ell, d_\ell]$, $(c' - c_\ell) \ge 2\eps(d_\ell - c_\ell)$, and $[c_\ell, d_\ell]$ is type (ii), then we merge $[c', d_k]$ and $[c_\ell, d_\ell]$ to form an interval $[c_\ell, d_k]$ of type (b)$. $
      To see this, first check  \eqref{eq: fsemiwellspaced}. For $x \in [c', d_k]$, by  our earlier calculation \eqref{align1} and $c_\ell \le c'$,
        \begin{equation}\label{align2}
            f(x) \ge f(d_k) - s (d_k - x) - \eps (d_k - c_\ell),
        \end{equation}
        and for $x \in [c_\ell, d_\ell]$,  by the fact that $(f, c_\ell, d_\ell)$ is $\eps^2$-superlinear with $s_f (c_\ell, d_\ell) = u$ and $d_k\geq d_l$, 
        \begin{align}
            f(x) &\ge f(c_\ell) + u(x - c_\ell) - \eps^2 (d_\ell - c_\ell) \label{align2.5}\\
            &\ge f(c_\ell) + u(x - c_\ell) - \eps^2 (d_k - c_\ell). \label{align3}
        \end{align}
        Now, we check that $s_f (c_\ell, d_k) \in [s, u]$. First, $s_f (c_\ell, d_k)$ is a convex combination of $s_f (c_\ell, d_\ell)$ and $s_f (d_\ell, d_k)$. Since $s_f (c_\ell, d_\ell) = u$ and $s_f (d_\ell, d_k) \le s$ by maximality of $c'$ and $d_\ell < c_k$, we have $s_f (c_\ell, d_k) \le u$.
        
        Likewise, $s_f (c_\ell, d_k)$ is also a convex combination of $s_f (c_\ell, c')$ and $s_f (c', d_k)$. Since $c' - c_\ell \ge 2\eps (d_\ell - c_\ell)$, \eqref{align2.5} implies that
        \begin{equation*}
            f(c') \ge f(c_\ell) + u(c' - c_\ell) - \frac{\eps}{2}(c' - c_\ell),
        \end{equation*}
        which means that $s_f (c_\ell, c') \ge u - \eps/2 \ge s$. Combined with $s_f (c', d_k) = s$, we get that $s_f (c_\ell, d_k) \ge s$. Hence, $s_f (c_\ell, d_k) \in [s, u]$.

        Finally, since $s_f (c_\ell, d_k) \in [s, u]$, equations \eqref{align2} and \eqref{align3} together tell us that $(f, c_\ell, d_k)$ is $\eps$-superlinear. Indeed, it suffices to check for $x\in [c', d_k]$, 
        \begin{align*}
        f(x) &\geq f(d_k)-s(d_k-x)-\eps(d_k-c_\ell)\\
        & \geq f(c_\ell)+ s_f(c_\ell, d_k)(d_k-c_\ell) -s(d_k-x) -\eps(d_k-c_\ell) \\
        &\geq s_f(c_{\ell}, d_k)(x-c_\ell) -\eps(d_k-c_\ell). 
        \end{align*}

        This shows the interval $[c_\ell, d_k]$ satisfies (1)(b), and then we can repeat the procedure with $k := \ell-1$.

   The procedure must terminate in $\le \tau^{-1}$ many steps and all resulting intervals satisfy one of the alternatives in (1). All of the resulting intervals have length $\ge 2\eps \tau m$, since they contain a least a proportion $2\eps$ of some $[c_j, d_j]$. Finally, the sum of their lengths is at least $(1 - O_{s,t,u} (\eps)) m$ since a proportion of at least $(1-2\eps)$ of each interval $[c_j, d_j]$ with $c_j \ge_{s,t} \eps m$ is preserved.  This gives the claim, with $\eps \tau$ in place of $\tau$.
    \end{itemize}
    
\end{proof}

As promised by Remark \ref{rem:dictionary}, here is a more precise version of what Proposition \ref{prop:multiscale} says about the set $\cP$ at different scales.
\begin{lemma}\label{lem:dictionary}
    Suppose  $n$ is a sufficiently large integer and $\Delta\in 2^{-\mathbb{N}}$ is sufficiently small. Let  $\cP\subset[0,1)^d$ is a $\{\Delta_j\}_{j=1}^n$-uniform set with branching function $f$ and $\Delta_j=\Delta^j$, and we apply Proposition~\ref{prop:multiscale} with $u = 2-s, m=n$. Fix $[c_j, d_j]\subset [0,m]$,  $\bp \in \cD_{\Delta^{c_j}}(\cP)$, $\cP_j := \phi_\bp (\cP \cap \bp)$.
    \begin{itemize}
        \item If  $[c_j, d_j]$ is a type (a) interval in Proposition~\ref{prop:multiscale} (1), then $\cP_j$ is a $(\Delta^{d_j-c_j}, t_j,  \Delta^{-(d_j-c_j)\eps-3})$-AD regular set.
        \item If $[c_j,d_j]$ is a type (b) interval in Proposition~\ref{prop:multiscale} (1), then $|\cP_j| = \Delta^{-(d_j - c_j) t_j}$ and $|\cP_j \cap Q| \lesim \Delta^{-(d_j - c_j) \eps-3} \cdot \max(r^{2-s} |\cP_j|, (r/\Delta^{d_j - c_j})^s)$ for any $r \in [\Delta^{d_j - c_j}, 1]$ and $Q \in \cD_r (\cP_j)$.
    \end{itemize} 
\end{lemma}

\begin{proof}
For any $0 \le k \le d_j - c_j$ and $Q \in \cD_{\Delta^k} (\cP_j)$, we get that
\begin{equation}\label{eqn:P cap Q}
    |\cP_j \cap Q|_{\Delta^{d_j-c_j}} = |\cP \cap \phi_\bp^{-1} (Q)|_{\Delta^{d_j}} = \frac{|\cP|_{\Delta^{d_j}}}{|\cP|_{\Delta^{c_j+k}}} = \Delta^{f(c_j + k) - f(d_j)}.
\end{equation}
In particular, when $k = 0$, we get $|\cP_j|_{\Delta^{d_j-c_j}} = \Delta^{f(c_j) - f(d_j)}$. Plugging back into \eqref{eqn:P cap Q} gives
\begin{equation}\label{eqn:P cap Q 2}
    |\cP_j \cap Q|_{\Delta^{d_j-c_j}} = \Delta^{f(c_j+k) - f(d_j)} = |\cP_j|_{\Delta^{d_j-c_j}} \cdot \Delta^{f(c_j+k) - f(c_j)}.
\end{equation}

We start by analyzing condition $(1)(a)$, which says that letting $t_j=s_f(c_j, d_j)$, 
\begin{equation}\label{eqn:eps linear specific}
    |f(x) - f(c_j) - t_j (x - c_j)| \le \eps (d_j - c_j), \qquad \forall x \in [c_j, d_j].
\end{equation}
We check that $\cP_j$ is a $(\Delta^{d_j-c_j}, t_j,  \Delta^{-(d_j-c_j)\eps -3})$-AD regular set. Fix $r \in (\Delta^{d_j-c_j}, 1)$ and $y \in \cP_j$; we wish to show
\[
\Delta^{\eps (d_j-c_j)} r^{t_j} |\cP_j|_{\Delta^{d_j-c_j}} \leq\Delta^{-3} |\cP_j \cap B(y, r)|_{\Delta^{d_j-c_j}} \le |\cP_j|_{\Delta^{d_j-c_j}} r^{t_j} \Delta^{-(d_j-c_j)\eps-3}
\]

For the upper bound, note that \eqref{eqn:P cap Q 2} and \eqref{eqn:eps linear specific} give us $|\cP_j \cap Q|_{\Delta^{d_j-c_j}} \le |\cP_j|_{\Delta^{d_j-c_j}} \cdot \Delta^{k t_j} \cdot \Delta^{-(d_j-c_j)\eps}$ for all $Q \in \cD_{\Delta^k} ([0,1]^2)$ (indeed, if $Q \notin \cD_{\Delta^k} (\cP_j)$, then the LHS is zero). Since any ball of radius $r \in (\Delta^{d_j-c_j}, 1)$ can be covered by $O(\Delta^2)$ many squares in $\cD_{\Delta^k} ([0,1]^2)$, where $k \in [0, d_j - c_j]$ is the largest integer with $r < \Delta^k$ (so $\Delta^k < r \Delta^{-1}$), we get $|\cP_j \cap B(y, r)| \lesim |\cP_j| \Delta^{-kt_j} \Delta^{-(d_j-c_j)\eps} \le |\cP_j| r^{t_j} \Delta^{-(d_j-c_j)\eps-3}$ when $\Delta$ is sufficiently small to absorb the implicit constant in $\lesssim$. 

The lower bound is similar: choose $k$ to be the smallest integer with $r > 2\Delta^k$ (so $2\Delta^{k-1} > r$). Then by applying \eqref{eqn:P cap Q 2} to $Q = \cD_{\Delta^k} (y)$ and \eqref{eqn:eps linear specific}, we verify the desired condition
\begin{equation*}
    \Delta^{\eps (d_j-c_j) +3} r^{t_j} |\cP_j|_{\Delta^{d_j-c_j}} \leq \Delta^{\eps (d_j-c_j)} \Delta^{k t_j} |\cP_j|_{\Delta^{d_j-c_j}} \le |\cP_j \cap Q|_{\Delta^{d_j-c_j}} \le |\cP_j \cap B(y, r)|_{\Delta^{d_j-c_j}}.
\end{equation*}
This completes the analysis of condition $(1)(a)$.

In a similar fashion, the condition in $(1)(b)$ becomes $f(d_j) - f(c_j) = t_j (d_j - c_j)$ and
\[
f(x)\geq \min ( f(c_j)+(2-s)(x-c_j), f(d_j) -s(d_j-x))-\eps (d_j-c_j), \qquad x \in [c_j, d_j].
\]
As before, from \eqref{eqn:P cap Q} with $k = 0$, we have $|\cP_j|_{\Delta^{d_j-c_j}} = \Delta^{f(c_j) - f(d_j)} = \Delta^{-(d_j - c_j) t_j}$. Finally, for a given $r \in [\Delta^{d_j-c_j}, 1]$, use \eqref{eqn:P cap Q}, \eqref{eqn:P cap Q 2}, and pick the largest $k \in [0, d_j - c_j]$ with $r < \Delta^k$ to obtain $|\cP_j \cap Q| \leq \Delta^{-(d_j - c_j) \eps -3} \cdot \max(r^{2-s} |\cP_j|, (r/\Delta^{d_j - c_j})^s)$ for any $Q \in \cD_r (\cP_j)$.

\end{proof}

We finally complete the proof of the Furstenberg set conjecture for general sets.

Recall Orponen-Shmerkin multi-scale incidence decomposition \cite[Corollary 5.37]{orponen2023projections}.

\begin{prop}\label{prop: OSincidence multiscale}
    Fix $N \ge 2$ and a sequence $\{ \Delta_j \}_{j=0}^n \subset 2^{-\N}$ with
    \begin{equation*}
        0 < \delta = \Delta_N < \Delta_{N-1} < \cdots < \Delta_1 < \Delta_0 = 1.
    \end{equation*}
    Let $(\cP_0, \cT_0) \subset \cD_\delta \times \T^\delta$ be a $(\delta, s, C, M)$-nice configuration. Then there exists a set $\cP \subset \cP_0$ such that:
    \begin{enumerate}
        \item $|\cD_{\Delta_j} (\cP)| \approx_\delta |\cD_{\Delta_j} (\cP_0)|$ and $|\cP \cap \textbf{p}| \approx_\delta |\cP_0 \cap \textbf{p}|$, $1 \leq j \leq N$, $\textbf{p} \in \cD_{\Delta_j} (\cP)$.
        \item For every $0 \leq j \leq N-1$ and $\textbf{p} \in \cD_{\Delta_j}$, there exist numbers $C_{\textbf{p}} \approx_\delta C$ and $M_{\textbf{p}} \geq 1$, and a family of tubes $\cT_{\textbf{p}} \subset \T^{\Delta_{j+1}/\Delta_j}$ with the property that $(\phi_{\textbf{p}} (\cP \cap \textbf{p}), \cT_{\textbf{p}})$ is a $(\Delta_{j+1}/\Delta_j, s, C_{\textbf{p}}, M_{\textbf{p}})$-nice configuration.
    \end{enumerate}

    Furthermore, the families $\cT_{\textbf{p}}$ can be chosen such that if $\textbf{p}_j \in \cD_{\Delta_j} (\cP)$ for $0 \le j \le N-1$, then
    \begin{equation*}
        \frac{|\cT_0|}{M} \geapp_\delta \prod_{j=0}^{N-1} \frac{|\cT_{\textbf{p}_j}|}{M_{\textbf{p}_j}}.
    \end{equation*}
    Here, $\geapp_\delta$ means $\gesim_N \log(1/\delta)^C$, and likewise for $\leapp_\delta, \app_\delta$.
\end{prop}

\begin{proof}[Proof of Theorem \ref{Thm: main}]
If $t \le s$ or $t \ge 2-s$ then we are done by \cite[Corollary 2.14]{orponen2021hausdorff} or Proposition \ref{prop: key} with $\Delta = \delta$ (see also \cite[Corollary 5.28]{orponen2023projections}), respectively. Hence, assume $t \in (s, 2-s)$.

We will eventually choose the following parameters: $\eta_0 (\eps, s), \tau (s, t, \eta_0), \eta (\eta_0, \tau), T(\eta)$.

For a fixed choice of $\eps, s$, let $\eta_0 (\eps, s) > 0$ be the minimum of $\eps$, $\frac{\eps^2}{1225}$ (in Corollary \ref{cor: key})  and the $\eta (\eps, s)$ in Proposition \ref{prop:furst-ADreg}. 

Let $\eta$ be a small constant to be chosen and  $T=T(\eta)$ such that $\frac{2\log T}{T}\leq \eta^2$. By \cite[Proposition A.1]{fassler2014restricted}, we can first refine $\cP$ to have cardinality $\lesim \delta^{-t}$. Refine $(\cP, \cT)$ such that $\cP$ is $\{2^{-jT}\}_{j=1}^m$-uniform for $2^{-mT}=\delta$ with associated sequence $\{N_j\}_{j=1}^m$. Let $f$ be the corresponding branching function. Since $\cP$ is a $(\delta, t, \delta^{-\eta})$-set, we have $f(x) \ge tx - \eta m$ for all $x \in [0, m]$. Thus, if we later choose $\eta < \eta_0^2$, then we may apply Proposition \ref{prop:multiscale} to the branching function $f$.

Let $\{ [c_j, d_j] \}_{j=1}^n$ be the intervals from Proposition \ref{prop:multiscale} applied with parameters $s, t, u := 2 - s, \eps := \eta_0$ to be specified later, corresponding to a sequence $0 < \delta = \Delta_n < \Delta_{n-1} < \cdots < \Delta_1 < \Delta_0 = 1$ (i.e. $\Delta_k$ is of the form $2^{-c_jT}$ or $2^{-d_jT}$). We can partition $\{ 0, 1, \cdots, n-1 \} = \cS \cup \cB$, ``structured'' $\cS$ (when $\Delta_k$ can be written as  $2^{-c_jT}$ for some $c_j$) and ``bad''  $\cB$ scales such that:
    \begin{enumerate}
        \item $\frac{\Delta_{j}}{\Delta_{j+1}} \ge \delta^{-\tau}$ for all $j \in \cS$, and $\prod_{j \in \cB} (\Delta_{j}/\Delta_{j+1}) \le \delta^{-O_{s,t} (\eta_0)}$;

        \item Using Lemma \ref{lem:dictionary}, we see that for each $j \in \cS$ and $\bp \in \cD_{\Delta_j} (\cP)$, we can find $t_j \in [s, 2-s]$ such that the set $\cP_j := \phi_{\bp} (\cP \cap \bp)$ is either
        \begin{enumerate} 
        \item[(i)] a $(\Delta_{j+1}/\Delta_j, t_j, 2^{-3T}(\Delta_j/ \Delta_{j+1})^{\eta_0})$-AD regular set; 
        \item[(ii)] satisfies $|\cP_j|_{\Delta_{j+1}/\Delta_j}  = \delta^{-t_j}$ and $|\cP_j \cap Q| \le 2^{-3T} (\Delta_j/\Delta_{j+1})^{\eta_0} \cdot \max\{ \rho^{2-s} |\cP_j|, \frac{\rho}{(\Delta_{j+1}/\Delta_j))^s}\}$ for any $\rho \in (\Delta_{j+1}/\Delta_j, 1)$ and $Q \in \cD_\rho (\cP_j)$.
        \end{enumerate}

        \item $\prod_{j \in \cS} (\Delta_j/\Delta_{j+1})^{t_j} \ge  |\cP|  \cdot \prod_{j \in \cB} (\Delta_{j+1}/\Delta_{j})^2 \ge \delta^{-t+\eta} \cdot \delta^{O_{s,t} (\eta_0)} = \delta^{-t + O_{s,t} (\eta_0)}$.
    \end{enumerate}
    Apply Proposition \ref{prop: OSincidence multiscale} and $\frac{\Delta_{j}}{\Delta_{j+1}} \ge \delta^{-\tau}$ to get a family of tubes $\cT_{\textbf{p}} \subset \T^{\Delta_{j+1}/\Delta_j}$ with the property that $(\cD_{\Delta_{j+1}} (\phi_{\textbf{p}} (\cP \cap \textbf{p})), \cT_{\textbf{p}})$ is a $(\Delta_{j+1}/\Delta_j, s, C_j, M_{\textbf{p}})$-nice configuration for some $C_j \lessapprox_\delta (\Delta_{j+1}/\Delta_j)^{-\tau^{-1} \eta}$ and
    \begin{equation*}
        \frac{|\cT_0|}{M} \geapp_\delta \prod_{j=0}^{N-1} \frac{|\cT_{\textbf{p}_j}|}{M_{\textbf{p}_j}}.
    \end{equation*}
    For $j \in \cS$, we get by either applying Proposition \ref{prop:furst-ADreg} for case (2)(i) or Corollary \ref{cor: key} for case (2)(ii), that if $\eta < \tau \eta_0 (\eps, s)$, then $\frac{|\cT_{\textbf{p}_j}|}{M_{\textbf{p}_j}} \ge (\Delta_j/\Delta_{j+1})^{\frac{s+t_j}{2}-\eps}$. Thus, since $\eta_0 \le \eps$, we have using items (1) and (3),
    \begin{align*}
        \frac{|\cT_0|}{M} & \geapp_\delta \prod_{j \in \cS} \left( \frac{\Delta_j}{\Delta_{j+1}} \right)^{\frac{s+t_j}{2}-\eps} \ge \prod_{j \in \cS} \left( \frac{\Delta_j}{\Delta_{j+1}} \right)^{\frac{s}{2}-\eps} \cdot \prod_{j \in \cS} \left( \frac{\Delta_j}{\Delta_{j+1}} \right)^{\frac{t_j}{2}} \\
        &\ge \delta^{(-1+O_{s,t} (\eta_0))(\frac{s}{2}- \eps)} \cdot \delta^{-\frac{t}{2} + O_{s,t} (\eta_0)}\\
        &\ge \delta^{-\frac{s+t}{2} + O_{s,t}(\eps)},
    \end{align*}
    which proves Theorem \ref{Thm: main}  with $O_{s,t}(\eps)$ in place of $\eps$.
\end{proof}

\begin{remark}
    After proving Theorem \ref{Thm: main}, a similar compactness argument as in the proof of Theorem \ref{prop:furst-ADreg} shows we can remove the dependence of $\eta$ on $t$. Indeed, we may in fact assume $t \in [s + \eps/100, 2 - s - \eps/100]$ because we can use \cite[Corollary 2.14]{orponen2021hausdorff} or \cite[Corollary 5.28]{orponen2023projections} for $t \le s + \frac{\eps}{100}$ or $t \ge 2 - s - \frac{\eps}{100}$, respectively.
\end{remark}

\bibliographystyle{plain}
\bibliography{journal2}
\end{document}